\documentclass[10pt]{article}

\usepackage{latexsym,amsfonts,amsmath,amsthm,amssymb, epsfig}
\usepackage[a4paper,total={170mm,250mm},vcentering,dvips]{geometry}
\usepackage{graphicx}
\usepackage{mathrsfs}
\usepackage[normalem]{ulem}
\usepackage[final]{changes}

\usepackage{yhmath}
\usepackage{dsfont}
\usepackage{enumitem}
\usepackage{csquotes}
\usepackage{bm}
\usepackage[pdftex,dvipsnames]{xcolor}
\usepackage{float}
\usepackage{caption}
\usepackage{parskip}
\usepackage{algpseudocode}
\usepackage{algorithm}
\usepackage{subcaption}

\usepackage[hidelinks]{hyperref}
\usepackage[noabbrev,capitalize]{cleveref}
\usepackage{tikz}
\usetikzlibrary{positioning}

\renewcommand{\autoref}{\cref}

\newtheorem{thm}{Theorem}[section]
\newtheorem{cor}[thm]{Corollary}

\newcommand{\rn}{\mathbb{R}^d}

\newcommand{\nat}{\mathbb{N}}
\newcommand{\whole}[1]{\ensuremath\lfloor #1 \rfloor}

\newcommand{\bit}{\begin{itemize}}
\newcommand{\eit}{\end{itemize}}

\providecommand{\U}[1]{\protect\rule{.1in}{.1in}}
\numberwithin{equation}{section}
\newtheorem{theorem}{Theorem}[section]

\newtheorem{corollary}[theorem]{Corollary}

\newtheorem{definition}[theorem]{Definition}
\newtheorem{lemma}[theorem]{Lemma}

\newtheorem*{proposition*}{Proposition}
\newtheorem{remark}[theorem]{Remark}
\newtheorem{example}[theorem]{Example}

\newtheorem{defn}[theorem]{Definition}
\newtheorem{prop}[theorem]{Proposition}

\newcommand{\bpr}{\begin{prop}}
\newcommand{\beq}{\begin{equation}}
\newcommand{\bd}{\begin{defn}}
\newcommand{\bex}{\begin{example}}
\newcommand{\bc}{\begin{cor}}
\newcommand{\bl}{\begin{lemma}}
\newcommand{\bt}{\begin{theorem}}
\newcommand{\br}{\begin{remark}}
\newcommand{\epr}{\end{prop}}
\newcommand{\eeq}{\end{equation}}
\newcommand{\ed}{\end{defn}}
\newcommand{\eex}{\end{example}}
\newcommand{\ec}{\end{cor}}
\newcommand{\el}{\end{lemma}}
\newcommand{\et}{\end{theorem}}
\newcommand{\er}{\end{remark}}

\begin{document}

\title{Fractional gradients and nonlocal PDEs in variable exponent spaces}
\author{
Susana Moura\footnote{Departmento de Matem\'atica, Universidade de Coimbra, Apartado 3008,
3001-454 Coimbra,  Portugal.
Email: \href{susana.moura@uc.pt}{susana.moura@uc.pt}} \  and
Cornelia Schneider\footnote{Friedrich-Alexander-Universit\"at Erlangen-N\"urnberg, Department of Mathematics, Cauerstr. 11, 91058 Erlangen, Germany.
Email: \href{mailto:cornelia.schneider@fau.de}{cornelia.schneider@fau.de}}.
}
\maketitle

\vskip1.5cm

\begin{abstract}
\noindent We develop a theory of distributional Riesz fractional gradients in variable
exponent Lebesgue spaces and introduce the associated nonlocal Sobolev spaces. Moreover, Sobolev, Hardy, Morrey, and Trudinger-type inequalities are established, and a
fractional fundamental theorem of calculus is derived. As an application, we study nonlinear
variational problems with nonstandard growth and establish existence and
Euler--Lagrange equations for minimizers.\\[0.1cm]
\end{abstract}

\section{Introduction}

Many phenomena in science and engineering involve interactions across multiple
length scales and cannot be adequately described by purely local models. This
has led to the development of nonlocal theories involving integro-differential
operators, which have found applications in materials science, diffusion
processes, image processing, and machine learning.
One important motivation for the present work comes from
\emph{peridynamics}, a nonlocal theory of continuum mechanics introduced by
Silling~\cite{Sil00}; see also the monographs
\cite{BFGS17,DEDG19,Ger16,MO14}. In contrast to classical elasticity, where
the behavior of a material is determined by local stress--strain relations,
peridynamics models the deformation of solids through interactions acting over
a finite horizon. This nonlocal description naturally accommodates
discontinuities such as cracks and therefore provides a powerful framework for
fracture mechanics and heterogeneous materials. While the classical
bond-based formulation of peridynamics, based on double-integral energies, is
among the most widely used nonlocal models, it does not fully capture the
structure of nonlinear elasticity. In particular, it was shown in
\cite{BCMC20,MS24} that this formulation is incompatible with nonlinear
elasticity in the nonlocal-to-local limit.
To overcome this limitation, Shieh and
Spector~\cite{SS15,SS18} introduced a new class of nonlocal models in which
the classical gradient is replaced by the \emph{Riesz fractional gradient}.
For sufficiently smooth functions
\(u:\mathbb{R}^d\to\mathbb{R}\), it is defined by
\[
\nabla^s u(x)
=
c_{d,s}
\int_{\mathbb{R}^d}
\frac{u(x)-u(y)}{|x-y|^{d+s}}
\frac{x-y}{|x-y|}
\,dy,
\qquad s\in(0,1).
\]
A remarkable feature of this operator is that it admits a distributional
formulation via duality, leading to Sobolev-type spaces that are equivalent to
the Bessel potential spaces \(L^{s,p}(\mathbb{R}^d)\). This provides a natural
framework for the analysis of nonlocal variational problems and fractional
partial differential equations; see, for example,
\cite{BCMC20b,CS19,KS22,Pon16,SS15,SS18}. \\
The aim of the present paper is to establish a theory of Riesz fractional
gradients in the setting of variable exponent Lebesgue spaces. Variable
integrability arises naturally in the modeling of heterogeneous materials,
whose mechanical response may vary from point to point. For instance, the
constitutive exponent may depend on a local damage variable according to
\[
p(x)=p_0-\alpha d(x),
\]
where \(d(x)\) describes the degree of material degradation. In undamaged
regions ($d(x)\approx 0$), the material behaves according to the exponent \(p_0\), whereas in
damaged regions the reduced exponent models a softer nonlinear response. This
naturally leads to fractional Sobolev spaces with variable integrability,
providing a flexible framework for nonlocal models with spatially
heterogeneous constitutive laws. \\
Fractional Sobolev spaces based on the Riesz fractional gradient have recently
been studied in generalized Orlicz spaces $L^A$ \cite{Campos24}, where they were
also identified with generalized Bessel--Orlicz potential spaces and
characterized as complex interpolation spaces \cite{CGS26}; see also
\cite{GS25} for the weighted case and \cite{OS26} for Bessel potential spaces
in classical Orlicz spaces. Variable exponent spaces correspond to the special
case $A(x,t)=t^{p(x)}$ (cf.\ \cite[Example~1]{Campos24}), so that some of our
results, in particular Theorems~\ref{thm-char-bessel} and~\ref{fractionalSI},
are contained in this more general theory. The variable exponent setting,
however, is considerably simpler and more explicit: the only structural
assumption is the log-H\"older continuity of $1/p(\cdot)$, and all exponents
are given pointwise in terms of $p(x)$. 
This
allows us to go beyond \cite{Campos24,CGS26}: we prove a fractional Hardy
inequality, Morrey- and Trudinger-type inequalities in the regimes $sp^->d$ and
$sp^-=d$, and existence of minimizers and Euler--Lagrange equations for general
convex integrands of $p(x)$-growth with nonhomogeneous exterior data. Existence results for linear problems with nonhomogeneous exterior data in weighted spaces were obtained in \cite[Thm.~5.3]{GS25}, and for quasilinear problems with zero exterior data in generalized Orlicz spaces in \cite[Thm.~6]{Campos24}; our Theorems~\ref{thm-minimizer} and~\ref{thm-Euler-Lagrange} combine nonlinear $p(x)$-growth, general convex integrands and nonhomogeneous exterior data.\\
Working directly in the variable exponent setting requires some care.
In contrast to the classical setting, variable-exponent spaces are naturally
described by modulars rather than norms, which necessitates different
techniques. A particular difficulty arises from the fact that Young's
convolution inequality does not directly extend to this framework, since the
exponent function \(p(\cdot)\) is not invariant under translations; see
\cite[Cor.~3.6.4]{DHHR} and the explanations given there. To overcome this obstacle, we employ a substitute result from
\cite[Lem.~4.6.3]{DHHR}, which applies only to convolutions with certain
classes of mollifiers. Consequently, throughout the analysis we have to verify
carefully that all convolutions satisfy the required assumptions. Furthermore, the treatment of nonlinear variational problems requires a
lower semicontinuity result in the variable-exponent setting. For the
reader's convenience, we include in Theorem~\ref{thm-low-semicont} a
self-contained proof of the weak lower semicontinuity of convex integral
functionals on $L^{p(\cdot)}$, based on Mazur's theorem. The argument is
classical and follows the lines of the corresponding results in Lebesgue
spaces, cf.\ \cite[Ch.~6]{FL07}.\\
We emphasize that the present paper should be viewed as a first step towards
the systematic study of generalized nonlocal fractional gradients and opens
several directions for further research. In particular, the case of
non-radial variable smoothness \(s(x)\) so far remains   unexplored. While
some initial investigations of the fractional Laplacian  and of
fractional gradients with radial variable smoothness have been carried out in
\cite{BMCS25,DaEtAl22}, a fundamental challenge in the non-radial setting is
the absence of an analogue of the Fourier transform providing a corresponding
spectral characterization; see Theorem~\ref{thm-prop-frac-grad}(iii). Another important direction concerns the localization of nonlocal gradients.
A limitation of the Riesz fractional gradient considered here, especially for
applications, is that its computation requires integration over the whole
space. To enable computations on bounded domains, which is essential, for
example, in realistic material models, several finite-horizon variants have
been introduced. These approaches localize the operator either by truncating
the Riesz fractional gradient or by incorporating spatially varying horizons
\(\delta(x)\); see
\cite{BellidoCuetoMoraCorral2023,CuetoKreisbeckSchoenberger2023,
KreisbeckSchoenberger2025,TianDu2017}. We plan to address these questions   in
future work.

This paper is structured as follows.
{Section \ref{sect-3} is devoted to the central functional framework, where we collect  the main preliminaries on variable Lebesgue spaces, define the fractional gradient, study the associated
nonlocal fractional Sobolev spaces with variable integrability,  and establish their principal structural properties.  Afterwards,} Section \ref{sect-4} applies these tools to the calculus of variations and PDEs, proving lower semicontinuity results and analyzing a nonlocal variational problem.

\section{Nonlocal fractional Sobolev spaces on variable Lebesgue spaces}
\label{sect-3}

\subsection{Preliminaries and variable Lebesgue spaces}
\label{sect-2}

We shall adopt the following general notation: $\nat$
denotes the set of all natural numbers, $\nat_0=\mathbb N\cup\{0\}$,
$\rn$, $d\in\nat$, denotes the $d$-dimensional real Euclidean space.
Furthermore, $\mu=|\cdot|$ stands for the Lebesgue measure.
For a
real number $a$, let $a_+:=\max(a,0)$ and let $\whole{a}$ denote its
integer part. For $p\in (0,\infty]$, the number $p'$ is defined by
$1/p':=(1-1/p)_+$ with the convention that $1/\infty=0$. By $c$,
$c_1$, $c_2$, etc. we denote positive constants independent of
appropriate quantities. For two non-negative expressions ({i.e.}
functions or functionals) ${\mathcal  A}$, ${\mathcal  B}$, the
symbol ${\mathcal A}\lesssim {\mathcal  B}$ (or ${\mathcal A}\gtrsim
{\mathcal  B}$) means that $ {\mathcal A}\leq c\, {\mathcal  B}$ (or
$c\,{\mathcal A}\geq {\mathcal B}$). If ${\mathcal  A}\lesssim
{\mathcal  B}$ and ${\mathcal A}\gtrsim{\mathcal  B}$, we write
${\mathcal  A}\sim {\mathcal B}$ and say that ${\mathcal  A}$ and
${\mathcal  B}$ are equivalent. Given two quasi-Banach spaces $X$
and $Y$, we write $X\hookrightarrow Y$ if $X\subset Y$ and the
natural embedding is bounded.

\paragraph{Variable Lebesgue spaces: Definition and basic properties}

Now we give a short survey on variable exponent Lebesgue spaces $L^{p(\cdot)}(\Omega)$, where  $\Omega$ is a  open  subset of $\rn$ with positive measure.
{The class of exponents $\mathcal{P}(\Omega)$ consists of all measurable functions $p:\Omega\rightarrow [1,\infty)$. }
 We denote $p^{+}:=\text{ess-sup}_{x\in \Omega}p(x)$ and $p^{-}:=\text{ess-inf}_{x\in \Omega}p(x)$.
 Moreover, the conjugate exponent $p'$ is defined by $\frac{1}{p(x)}+\frac{1}{p'(x)}=1$, where $1/\infty:=0$.
 The variable exponent Lebesgue space $L^{p(\cdot)}(\Omega)$
consists of all measurable functions $f$ such that for some $\lambda>0$ the modular $\varrho_{L^{p(\cdot)}(\Omega)}(f/\lambda)$ is finite, where
$$
\varrho_{L^{p(\cdot)}(\Omega)}(f):=\int_{\Omega} |f(x)|^{p(x)}\, dx.
$$
The Luxembourg norm of a function $f\in L^{p(\cdot)}(\Omega)$ is given by
$$
\|f\|_{L^{p(\cdot)}(\Omega)}:=\inf\left\{\lambda>0:\varrho_{L^{p(\cdot)}(\Omega)}\left(\frac{f}{\lambda}\right)\leq 1 \right\}.
$$

{
In the case of non-constant exponent $p(\cdot)$ we  require certain regularity conditions on $p(\cdot)\in \mathcal{P}(\Omega)$ in order to
 get similar results  as in the constant case, e.g., in order to establish  the boundedness of the maximal operator
 in $L^{p(\cdot)}(\Omega)$.}\\

\begin{definition}\label{log-holder}
Let $g\in C(\Omega)$. We say that $g$ is locally log-H\"older continuous, abbreviated $g\in C_{\mathrm{loc}}^{\log}(\Omega)$, if there exists $c_{\log}>0$ such that
$$
|g(x)-g(y)|\leq \frac{c_{\log}}{\log(e+1/|x-y|)} \quad \text{for all}\;\; x,y\in\Omega.
$$
We say that $g$ is globally log-H\"older continuous, abbreviated $g\in C^{\log}(\Omega)$,
if $g$ is locally log-H\"older continuous and there exists $g_{\infty}\in\mathbb{R}$ such that
$$
|g(x)-g_{\infty}|\leq \frac{c_{\log}}{\log(e+|x|)} \quad \text{for all}\;\; x\in\Omega.
$$
\end{definition}

We use the notation $p(\cdot)\in \mathcal{P}^{\log}(\Omega)$ to mean that $p(\cdot)\in \mathcal{P}(\Omega)$ and $1/p(\cdot)\in C^{\log}(\Omega)$.
It was proved in \cite{DHHMS09} that the maximal operator is bounded in $L^{p(\cdot)}(\rn)$ provided that
$p(\cdot)\in \mathcal{P}^{\log}(\rn)$ and $1<p^-\leq p^+\leq \infty$.

\subsection{The fractional gradient and its properties}

 For $0<\alpha<d$ and $u\in\mathcal{S}(\mathbb{R}^d)$, the
\emph{Riesz potentials} of order $\alpha$ are defined  by the formula
\begin{equation}
I_\alpha u:=I_\alpha*u, \qquad
 I_\alpha(x):=
\frac{\gamma(d,\alpha)}{|x|^{d-\alpha}},
\end{equation}
with  constant
\(
\gamma(d,\alpha):=
\frac{\Gamma\!\left(\frac{d-\alpha}{2}\right)}
{\pi^{d/2}2^\alpha
\Gamma\!\left(\frac{\alpha}{2}\right)},
\)
{ where $\Gamma$ denotes the gamma function. This normalizing constant is chosen so  that  $\widehat{I_{\alpha}}(\xi)=(2\pi |\xi|)^{-\alpha}$ in $\mathcal{S}'(\rn)$. Here, the Fourier transform is given by
$$
\widehat{f}(x)= \int_{\rn} e^{-2\pi i x\cdot y} f(y)\, dy, \qquad f\in
\mathcal{S}(\rn),$$
and extended by duality to $\mathcal{S}'(\rn)$.}
Regarding { the Riesz} potentials, among their fundamental properties are the
semigroup property
\begin{equation}
I_\alpha I_\beta u
=
I_{\alpha+\beta}u,
\qquad
\alpha,\beta>0,\quad
\alpha+\beta<d,
\end{equation}
and the fact that the Laplacian maps a potential of order
$\alpha+2$ to one of order $\alpha$,
\begin{equation}
-\Delta I_{\alpha+2}u
=
I_\alpha u.
\end{equation}

Through analytic continuation, the Riesz potential can be extended to
negative exponents (see Chapter~I, p.~41 in \cite{Landkof1973}), which leads to the representation of the fractional Laplacian,
\begin{equation}
(-\Delta)^s u
:=
I_{-2s}u.
\end{equation}

In analogy to \cite{SS15}, we now define the fractional Riesz gradient with the help of these Riesz potentials.

\begin{definition}[Fractional gradient]
Let $s\in(0,1)$. If $u\in L^{p(\cdot)}(\mathbb{R}^d)$ for some $p\in \mathcal{P}^{\log}(\rn)$ with $1<p^-\leq p^+< \infty$ is such that $I_{1-s}*u$   is a well-defined locally integrable function, we define the distributional Riesz fractional gradient
\[
(\nabla^s u)_j := \frac{\partial^s u}{\partial x_j^s},
\qquad
j=1,\ldots,d,
\]
where
\begin{equation}
\frac{\partial^s u}{\partial x_j^s}
:=
\frac{\partial}{\partial x_j} I_{1-s}u,
\end{equation}
in the sense that
\begin{align*}
\left\langle
\frac{\partial^s u}{\partial x_j^s},
v
\right\rangle
&=
(-1)
\left\langle
I_{1-s}u,
\frac{\partial v}{\partial x_j}
\right\rangle
=
-\int_{\mathbb{R}^d}
\bigl(I_{1-s}*u\bigr)
\frac{\partial v}{\partial x_j}
\,dx,
\end{align*}
for every $v\in C_0^\infty(\mathbb{R}^d)$.
\end{definition}

\begin{remark}\label{rem:bd-Riesz} Let us shed some light on  what it  means for $I_{1-s}\ast u$ to be well defined in the setting of variable Lebesgue spaces.
The $L^{p(\cdot)}(\rn)\rightarrow L^{q(\cdot)}(\rn)$ boundedness of $I_{\alpha}$  was first considered in \cite{Sam98}. For the case of the whole
space $\rn$ the following statement was proven in \cite{CCUF07} and can also be found in \cite[Thm.~2.2]{AS06} and \cite[Thm.~6.1.9]{DHHR}.
Let \(0<\alpha<d\) and let \(p(\cdot)\in \mathcal{P}^{\log}(\rn)\) such that
\(
1<p^- \le p^+ < \frac{d}{\alpha}.
\)
Then there exists a constant \(c>0\) such that
\begin{equation}\label{eq:bd-Riesz}
\|{I}_{\alpha}f\|_{L^{q(\cdot)}(\mathbb{R}^d)}
\le
c\,\|f\|_{L^{p(\cdot)}(\mathbb{R}^d)},
\qquad
f\in L^{p(\cdot)}(\mathbb{R}^d),
\end{equation}
where \(q(\cdot)\) is  given by
\(
\frac{1}{q(x)}
=
\frac{1}{p(x)}
-
\frac{\alpha}{d},
\
x\in\mathbb{R}^d.
\)
In particular, under the hypothesis $p^+<\tfrac{d}{1-s}$, \eqref{eq:bd-Riesz} shows that $I_{1-s}u\in L^{q(\cdot)}(\mathbb R^d)$ with $q^->p^->1$, which in particular guarantees that $I_{1-s}u$ is locally integrable, as required in the definition above.
\end{remark}

In \cite{SS15} the following properties analogous to the gradient were shown.

\begin{theorem}[Basic properties of the Riesz fractional gradient]\label{thm-prop-frac-grad}
Let $s\in(0,1)$. Then the Riesz fractional gradient satisfies the following properties.

\begin{itemize}
    \item[(i)] If $u\in C_0^\infty(\mathbb{R}^d)$, then
    \[
        \nabla^s u = I_{1-s}*\nabla u.
    \]

    \item[(ii)] If $u\in C_0^\infty(\mathbb{R}^d)$, then the fractional Laplacian admits the representation
    \[
        (-\Delta)^s u
        =
        -\sum_{j=1}^d
        \frac{\partial^s}{\partial x_j^s}
        \frac{\partial^s}{\partial x_j^s}u.
    \]

    \item[(iii)] If $u\in\mathcal{S}(\mathbb{R}^d)$, then the Fourier transform of the fractional partial derivatives is given by
    \[
        \widehat{\frac{\partial^s u}{\partial x_j^s}}(\xi)
        =
        (2\pi)^s i\xi_j |\xi|^{-1+s}\widehat{u}(\xi).
    \]
\end{itemize}
\end{theorem}

\begin{proof} We provide a proof of (iii) since there appears to be a sign inconsistency  in \cite[Thm.~1.4]{SS15}. We compute
\begin{align*}
\widehat{\frac{\partial^s u}{\partial x_j^s}}(\xi)
=
\widehat{\frac{\partial}{\partial x_j}(I_{1-s}u)}(\xi)
=
2\pi i\xi_j\,\widehat{I_{1-s}* u}(\xi)
=
2\pi i\xi_j\,(2\pi)^{s-1}|\xi|^{s-1}\widehat{u}(\xi)
=
i(2\pi)^s|\xi|^{s-1}\xi_j\,\widehat{u}(\xi).
\end{align*}
\end{proof}

\begin{remark}
The first identity provides a geometric interpretation of the fractional gradient. While the classical gradient is the locally linear approximation of a function, the fractional gradient can be viewed as an \emph{averaged} linear approximation, incorporating nonlocal information through the Riesz potential. The second identity establishes the natural connection between the fractional gradient and the fractional Laplacian. It is the nonlocal analogue of the classical identity
\[
    -\Delta u
    =
    -\sum_{j=1}^d
    \frac{\partial^2u}{\partial x_j^2},
\]
showing that the fractional Laplacian can be recovered as the sum of the fractional second-order derivatives. Finally, the Fourier representation provides the spectral characterization of the fractional gradient. Together with the previous identities, it links the geometric and spectral viewpoints.
\end{remark}

\subsection{Nonlocal Sobolev spaces and Bessel potential spaces with variable integrability}

The above observations motivate the introduction of nonlocal fractional Sobolev spaces as the natural energy spaces for nonlocal variational problems and fractional PDEs.

\begin{definition} \label{def-Xsp}
Let $p(\cdot)\in\mathcal{P}^{\log}(\mathbb{R}^d)$ and let $s\in(0,1)$. For
$u\in C_0^\infty(\mathbb{R}^d)$ we define the norm
\[
    \|u\|_{X^{s,p(\cdot)}(\mathbb{R}^d)}
    :=
    \|u\|_{L^{p(\cdot)}(\mathbb{R}^d)}
    +
    \| \nabla^s u\|_{L^{p(\cdot)}(\mathbb{R}^d;\mathbb{R}^d)},
\]
where the latter norm is understood as the  norm  in $L^{p(\cdot)}(\mathbb{R}^d)$
of the Euclidean norm of $\nabla^s u$, i.e., $|\nabla^s u|=(\sum_{j=1}^d|\partial^s_{x_j}u|^2)^{1/2}$.
The space $X^{s,p(\cdot)}(\mathbb{R}^d)$ is then defined by
\[ X^{s,p(\cdot)}(\mathbb{R}^d)
    :=
    \overline{C_0^\infty(\mathbb{R}^d)}
    ^{\|\cdot\|_{X^{s,p(\cdot)}(\mathbb{R}^d)}},
\]
the closure of $C_0^\infty(\mathbb{R}^d)$ with respect to the norm of $X^{s,p(\cdot)}(\mathbb{R}^d)$.
\end{definition}

\begin{remark}
We restrict our attention to the case $s\in(0,1)$, although the construction extends naturally to arbitrary orders $s\ge0$. Writing
\(
    \tilde s=s-\lfloor s\rfloor,
\)
where $\lfloor s\rfloor$ denotes the integer part of $s$, one may define
\[
    \|u\|_{X^{s,p(\cdot)}(\mathbb{R}^d)}
    :=
    \|u\|_{W^{\lfloor s\rfloor,p(\cdot)}(\mathbb{R}^d)}
    +
    \|\nabla^{\tilde s}D^{\lfloor s\rfloor}u\|_{L^{p(\cdot)}(\mathbb{R}^d;{\rn})},
\]
where $W^{\lfloor s\rfloor,p(\cdot)}(\mathbb{R}^d)$ stands for the Sobolev space of order $\lfloor s\rfloor$ with variable exponent $p(\cdot)$.
\end{remark}

As an important tool  we derive  a fractional version of the fundamental theorem of calculus, which we use in the sequel for proving several results.

\begin{theorem}[Fractional FTOC]\label{thm-frac-FTOC}
Let $0<s<1$, then  every \( u \in C_0^{\infty}(\mathbb{R}^d) \) is expressible as
\[
u = I_s \sum_{j=1}^d \mathcal{R}_j \frac{\partial^s u}{\partial x_j^s},
\]
where \( \mathcal{R}_j \) is the Riesz transform, which can be characterized as a singular integral or \( 0 \)-order operator with symbol \( \frac{-i\xi_j}{|\xi|} \), i.e.,  $\widehat{\mathcal{R}_jf} =\frac{-i\xi_j}{|\xi|}\hat{f}$ for $f\in \mathcal{S}(\rn)$.
\end{theorem}

\begin{proof}
 Let $u \in C_0^{\infty}(\mathbb{R}^d)$.  Since $u\in \mathcal{S}(\rn)$,
using
Theorem \ref{thm-prop-frac-grad}(iii), which gives
$
\widehat{\frac{\partial^s u}{\partial x_j^s}} =  (2\pi)^s i \xi_j |\xi|^{-1+s} \hat{u},
\  \text{and} \
\widehat{\mathcal{R}_jf} = -\frac{i \xi_j}{|\xi|}\hat{f}
$,
the result follows from the identity
\(
\hat{I}_s = (2\pi |\xi|)^{-s} \ \text{in } \mathcal{S}^\prime(\mathbb{R}^d)
\),  as
\[
\hat{u}=\hat{I}_s \sum_{j=1}^d -\frac{i \xi_j}{|\xi|}\widehat{\frac{\partial^s u}{\partial x_j^s}}.
\]
\end{proof}

Having introduced { the space $X^{s,p(\cdot)}(\mathbb{R}^d)$}, it is natural to ask whether it coincides with the corresponding variable-exponent Bessel potential space.
{In the constant exponent setting, the coincidence was established in  \cite[Thm. 1.17]{SS15}.}
An analogous characterization in the variable exponent setting can be shown, cf. Theorem \ref{thm-char-bessel} below.
{We first recall the definition of variable Bessel potential spaces, following \cite{AS06} (see also  \cite{GHN07}). The Bessel kernel $g_s$ of order $s$, $s>0$, is defined by
\[\widehat{g_s}(\xi):=(1+4\pi^2|\xi|^2)^{-s/2}, \qquad \xi\in \rn. \]
It is known that
$$ {g_s}(x)= \frac{1}{(4\pi)^{\frac{s}{2}}\Gamma(\frac{s}{2})} \int_0^\infty e^{-\frac{\pi|x|^2}{\delta}} e^{-\frac{\delta}{4\pi}} \delta^{\frac{s-d}{2}} \, \frac{d\delta}{\delta}, \qquad x\in \rn,
$$
cf. e.g. \cite[7.1, p.~245]{Miz96}.
For $p(\cdot) \in \mathcal P^{\log}(\mathbb R^d)$ with  $1<p^- \leq  p^+<\infty$ and  $s> 0$,  the Bessel potential space with variable exponent $L^{s,p(\cdot)}(\mathbb R^d)$ is defined by
\[ L^{s,p(\cdot)}(\mathbb R^d) :=
\left\{u=g_{s}* f:\, f\in L^{p(\cdot)}(\mathbb R^d)\right\}, \]
endowed with the norm
\[ \|u\|_{L^{s,p(\cdot)}(\mathbb R^d)} = \bigl\| f \bigr\|_{L^{p(\cdot)}(\mathbb R^d)}, \qquad u= g_s *f. \]
}

\begin{theorem}[Characterization via  variable Bessel potential spaces]\label{thm-char-bessel}
Assume that $p(\cdot)\in\mathcal{P}^{\log}(\mathbb{R}^d)$ satisfies
\(
1<p^-\le p^+<\infty\)   and let  $s\in (0,1)$.
Then
\[
    X^{s,p(\cdot)}(\mathbb{R}^d)
    =
    L^{s,p(\cdot)}(\mathbb{R}^d).
\]
\end{theorem}

\begin{proof}
We first prove the inclusion \( L^{s,p(\cdot)}(\mathbb{R}^d) \subset X^{s,p(\cdot)}(\mathbb{R}^d) \).
Suppose \( u \in L^{s,p(\cdot)}(\mathbb{R}^d) \), which means that \( u = g_s * f \) for some \( f \in L^{p(\cdot)}(\mathbb{R}^d) \).
Then using \cite[Lem.~12.4.3]{DHHR}
(see also  \cite[Thm. 5.1]{AS06})
we have
\begin{equation}\label{help-1}
\|u\|_{L^{p(\cdot)}(\mathbb{R}^d)} \leq c \|g_s\|_{L^1(\mathbb{R}^d)} \|f\|_{L^{p(\cdot)}(\mathbb{R}^d)},
\end{equation}
and thus \( u \in L^{p(\cdot)}(\mathbb{R}^d) \). It therefore remains to show that \( {\nabla^s u } \in L^{p(\cdot)}(\mathbb{R}^d;\mathbb{R}^d) \) and the inclusion will be demonstrated. We begin by assuming \( f \in C_0^\infty(\mathbb{R}^d) \), and then the general case will follow from density. Note that \( f \in C_0^\infty(\mathbb{R}^d) \) and \( u \in L^{p(\cdot)}(\mathbb{R}^d) \) implies that it makes sense to write

\[
\frac{\partial^s u}{\partial x_j^s} = g_s * \frac{\partial^s f}{\partial x_j^s}.
\]
Then, using Theorem \ref{thm-prop-frac-grad}(iii) to compute the Fourier transform, we obtain
\[
\widehat{\frac{\partial^s u}{\partial x_j^s}}(\xi) = (1 + 4\pi^2 |\xi|^2)^{-\frac{s}{2}} ((2\pi)^s i \xi_j |\xi|^{-1+s}) \hat{f}(\xi),
\]
so that
\[
\widehat{\frac{\partial^s u}{\partial x_j^s}} = \frac{i \xi_j}{|\xi|} \frac{(2\pi |\xi|)^s}{(1 + 4\pi^2 |\xi|^2)^{\frac{s}{2}}} \hat{f}.
\]

On the other hand, it follows from \cite[Lem.~2.1, p.~247]{Miz96} (see also \cite[Lem.~2, p.~133]{Stein70}) that there exists a finite signed measure $\mu_s$ such that
\[
\hat{\mu}_s(\xi) = \frac{(2\pi |\xi|)^s}{(1 + 4\pi^2 |\xi|^2)^{\frac{s}{2}}},
\]
and therefore inverting the Fourier transform we have
\[
\frac{\partial^s u}{\partial x_j^s} = { -}
\mu_s * \mathcal{R}_j f,
\]
where \( \mathcal{R}_j \) is the Riesz transform. Now
\begin{equation}\label{help-2}
\|\nabla^s u\|_{L^{p(\cdot)}(\mathbb{R}^d;{\rn})} \leq c{ \max_{j=1,\dots, d}}\|\mathcal{R}_j f\|_{L^{p(\cdot)}(\mathbb{R}^d)} \leq c' \|f\|_{L^{p(\cdot)}(\mathbb{R}^d)},
\end{equation}
where
in the second step we used the
  boundedness of the Riesz transform \( \mathcal{R}_j : L^{p(\cdot)}(\mathbb{R}^d) \to L^{p(\cdot)}(\mathbb{R}^d) \) for \( 1 < p^- \leq p^+ < \infty \), which also is stated in \cite[Lem.~12.4.3]{DHHR}. Now \eqref{help-1} together with \eqref{help-2} implies the desired result  is case of \( f \in C_0^\infty(\mathbb{R}^d) \). Otherwise it follows by density, since \(C_0^\infty(\mathbb{R}^d) \) in dense in \(L^{p(\cdot)}(\mathbb{R}^d)\), cf. e.g \cite[Lem.~4.6.5, p.~130]{DHHR}.

Conversely, suppose \( u \in X^{s,p(\cdot)}(\mathbb{R}^d) \). Define
\[
f := \lambda_s * \left( u + \sum_{j=1}^d \mathcal{R}_j \frac{\partial^s u}{\partial x_j^s} \right),
\]
where \( \lambda_s \) is the finite signed measure for which
\[
(1 + 4\pi^2 |\xi|^2)^{\frac{s}{2}} = \hat{\lambda}_s(\xi) [1 + (2\pi |\xi|)^s],
\]
cf.  \cite[Lem.~2.2, p. 248]{Miz96}.
Moreover, we introduce a function $h_{s}$ via its Fourier transform as follows
\[
\frac{(1 + 4\pi^2 |\xi|^2)^{\frac{s}{2}}}{1 + (2\pi |\xi|)^s}=\hat{\lambda}_s(\xi)=:\hat{h}_s(\xi) +1.
\]
In  particular, this implies $\lambda_s=h_s+\delta_0$, where $\delta_0$ denotes the delta distribution.
It was shown in \cite[Lem.~5.4]{AS06} that $h_s$ admits the following bounds:
\[
|h_s(x)| \leq C
\begin{cases}
|x|^{-(d-s)}, & |x| < 1 \\
|x|^{-(d+s)}, & |x| \geq 1
\end{cases}.
\]
In what follows we will show that \( |h_s|\) has an integrable least bell shaped majorant, i.e., the function  \(\Psi_s(x) := \sup_{|y| \geq |x|} |h_s(y)|\in L^1(\rn)\), so that \cite[Lem.~4.6.3]{DHHR} can be applied w.r.t. $h_s$. For
  \(|x| < 1\),
\[
\Psi_s(x) \leq \max\left\{ \sup_{|x| \leq |y| < 1} C|y|^{-(d-{s})},\, \sup_{|y| \geq 1} C|y|^{-(d+s)} \right\}
= C |x|^{-(d-{s})}.
\]
For \(|x| \geq 1\),
\[
\Psi_s(x) = \sup_{|y| \geq |x|} |h_s(y)| \leq C |x|^{-(d+s)}.
\]
Thus,
\[
\Psi_s(x) \leq C
\begin{cases}
|x|^{-(d-s)}, & |x| < 1, \\
|x|^{-(d+s)}, & |x| \geq 1.
\end{cases}
\]
Using spherical coordinates,
\[
\int_{|x| < 1} |x|^{-(d-{s})} \, dx
= C_d \int_0^1 r^{{s}-1} \, dr < \infty,
\]
and
\[
\int_{|x| \geq 1} |x|^{-(d+s)} \, dx
= C_d \int_1^\infty r^{-(s+1)} \, dr < \infty,
\]
since \(s > 0\). Hence \(\Psi_s \in L^1(\mathbb{R}^n)\).
We compute
\begin{align*}
    \|f\|_{L^{p(\cdot)}(\mathbb R^d)}
    & =\left\|\lambda_s * \left( u + \sum_{j=1}^d \mathcal{R}_j \frac{\partial^s u}{\partial x_j^s} \right)\right\|_{L^{p(\cdot)}(\mathbb R^d)}
     =\left\|(h_s+\delta_0) * \left( u + \sum_{j=1}^d \mathcal{R}_j \frac{\partial^s u}{\partial x_j^s} \right)\right\|_{L^{p(\cdot)}(\mathbb R^d)}\\
     & \le \left\|h_s * \left( u + \sum_{j=1}^d \mathcal{R}_j \frac{\partial^s u}{\partial x_j^s} \right)\right\|_{L^{p(\cdot)}(\mathbb R^d)}+ \left\| \delta_0 * \left( u + \sum_{j=1}^d \mathcal{R}_j \frac{\partial^s u}{\partial x_j^s} \right)\right\|_{L^{p(\cdot)}(\mathbb R^d)}\\
    &\leq c\left\| u + \sum_{j=1}^d \mathcal{R}_j \frac{\partial^s u}{\partial x_j^s}\right\|_{L^{p(\cdot)}(\mathbb R^d)}
    \leq C \left(\|u\|_{L^{p(\cdot)}(\mathbb R^d)}+\sum_{j=1}^d \left\|\frac{\partial^s u}{\partial x_j^s}\right\|_{L^{p(\cdot)}(\mathbb R^d)}\right),
\end{align*}
where we used the definition of $h_s$ in the second step, the identity $\delta_0 \ast f =f$    and \cite[Lem.~12.4.3]{DHHR}
in the fourth step, as well as  the boundedness of the  Riesz transform  in the last step (see, e.g., \cite[Lem.~3.3]{GHN07} or \cite[Lem.~12.4.3]{DHHR}).
The right-hand side is bounded since $u\in X^{s,p(\cdot)}(\rn)$ which shows that \( f \in L^{p(\cdot)}(\mathbb{R}^d) \).
It remains to show that \( u = g_s * f \). Supposing that \( u \in \mathcal{S}(\mathbb{R}^d) \cap X^{s,p(\cdot)}(\mathbb{R}^d) \), we have that
\[
\hat{g}_s \hat{f} = \hat{g}_s \hat{\lambda}_s \left( \hat{u} + \sum_{j=1}^d \frac{-{i}{\xi}_j}{|\xi|} ((2\pi)^s { i {\xi}_j} |\xi|^{-1+s}) \hat{u} \right),
\]
where we used Theorem \ref{thm-prop-frac-grad}(iii) and the fact that $\widehat{\mathcal{R}_j\varphi} (\xi)= -\frac{i \xi_j}{|\xi|}\hat{\varphi}(\xi)$, $\varphi\in \mathcal{S}(\rn)$.
Making simplifications in the summation on the right hand side, we have
\[
\hat{g}_s \hat{f} = \hat{g}_s \hat{\lambda}_s (\hat{u} + (2\pi |\xi|)^s \hat{u}) = \hat{u},
\]
where we  used the definition of \( \lambda_s \)  and  $g_s$ in the second equality.
Therefore \( u = g_s * f \) whenever \( u \in \mathcal{S}(\mathbb{R}^d) \). The general case then follows from the density of \( \mathcal{S}(\mathbb{R}^d) \) in \( X^{s,p(\cdot)}(\mathbb{R}^d) \), since \( C_0^\infty(\mathbb{R}^d) \subset \mathcal{S}(\mathbb{R}^d) \cap X^{s,p(\cdot)}(\mathbb{R}^d) \subset X^{s,p(\cdot)}(\mathbb{R}^d) \) and \( X^{s,p(\cdot)}(\mathbb{R}^d) \) is defined as the completion of \( C_0^\infty(\mathbb{R}^d) \) in the \( X^{s,p(\cdot)} \)-norm.
\end{proof}

\begin{remark} An important consequence of Theorem \ref{thm-char-bessel} is that
we can deduce many properties about $X^{s,p(\cdot)}(\mathbb{R}^d)$ from the existing literature on Bessel potential spaces. Moreover, a
 characterization closely related to Theorem~\ref{thm-char-bessel} was proved in \cite{AS06} for a scale of spaces defined analogously to those in Definition~\ref{def-Xsp}, with the fractional Riesz derivative replacing the fractional gradient.  Specifically, if $0<\alpha<d$ and $p(\cdot)\in\mathcal{P}^{\log}(\rn)$, and if $Y^{\alpha,p(\cdot)}(\rn)$ denotes the space of all functions $f\in L^{p(\cdot)}(\rn)$ such that $D^\alpha f\in L^{p(\cdot)}(\rn)$, where $D^\alpha f$ is the fractional Riesz derivative of order $\alpha$, then \[ L^{\alpha,p(\cdot)}(\rn)=Y^{\alpha,p(\cdot)}(\rn), \] whenever $1<p^- \le p^+ < d/\alpha$.
Theorem~\ref{thm-char-bessel} is also a special case of \cite[Thm.~4.5]{CGS26}.
\end{remark}

\subsection{Properties of variable nonlocal fractional Sobolev spaces}

\begin{theorem}[Sobolev-type inequality]\label{fractionalSI}
Assume that $p(\cdot)\in\mathcal{P}^{\log}(\mathbb{R}^d)$ satisfies
\(
1<p^- \le p^+<\infty
\),
and let $s\in(0,1)$ be such that $s<\frac{d}{p^+}$. Then there exists
$C=C(d,p,s)>0$ such that
\[
\|u\|_{L^{p^{*}(\cdot)}(\mathbb{R}^{d})}
\leq
C\,\|\nabla^{s}u\|_{L^{p(\cdot)}(\mathbb{R}^{d};\rn)}
\]
for all $u\in {X}^{s,p(\cdot)}(\mathbb{R}^{d})$, where for $x\in \rn$,
\begin{equation}\label{def-p-ast}
\frac{1}{p^{*}(x)}:=\frac{1}{p(x)}-\frac{s}{d}.
\end{equation}
\end{theorem}

\begin{proof}
The result follows in a parallel way to the proof  of Theorem 1.8 in \cite{SS15}, where in replacement of Lemma 2.4 therein we now use \eqref{eq:bd-Riesz} from Remark \ref{rem:bd-Riesz}.
Alternatively, the result can also be deduced from the
Sobolev inequality in generalized Orlicz spaces \cite[Thm.~1]{Campos24} with
$A(x,t)=t^{p(x)}$ and $B(x,t)=t^{p^*(x)}$.
\end{proof}

\begin{theorem}[Local Sobolev type inequality] \label{Iq}
Let $p(\cdot)\in\mathcal{P}^{\log}(\mathbb{R}^d)$ be such that \(
1<p^- \le p^+<\infty\), $s\in(0,1)$, $s<\frac{d}{p^+}$, define $p^{\ast}$ as in \eqref{def-p-ast} above. {  Let $\Omega \subset \rn$ be a bounded domain}  and
assume  that $q(\cdot)\in  {\mathcal{P}}({\Omega})$  satisfies  $p^*(x)>q(x)$ for every $x\in \Omega$ and $q^-\ge 1$.
Then,
there exists a positive constant $C=C(\Omega,d,s,p,q)$ such that
\[
\|u\|_{L^{q(\cdot)}(\Omega)}
\le
C\,\|\nabla^s u\|_{L^{p(\cdot)}(\mathbb{R}^d;{\rn})}
\]
for all
\(
u\in {X}^{s,p(\cdot)}(\mathbb{R}^d).
\)
\end{theorem}

\begin{proof}
Follows from H\"older's inequality, the boundedness of $\Omega$ and Theorem \ref{fractionalSI}. Indeed, with
\[
\frac{1}{r(x)}
:=
\frac{1}{q(x)}-\frac{1}{p^*(x)},
\]
it holds
\[
\|u\|_{L^{q(\cdot)}(\Omega)}
\le c \|u\|_{L^{p^{*}(\cdot)}(\rn)} \|\chi_{\Omega}\|_{L^{r(\cdot)}(\rn)} \le
c\, \max\{|\Omega|^{1/r^-},|\Omega|^{1/r^+}\}\,\|\nabla^s u\|_{L^{p(\cdot)}(\mathbb{R}^d;{\rn})}
\le C \|\nabla^s u\|_{L^{p(\cdot)}(\mathbb{R}^d;{\rn})}.
\]
Alternatively, Theorem \ref{Iq} can be deduced from the
Sobolev inequality in generalized Orlicz spaces \cite[Thm.~1]{Campos24}.
\end{proof}

Moreover, we also have a fractional Hardy inequality in the setting of variable Lebesgue spaces.

\begin{theorem}[Fractional Hardy inequality]\label{thm:SteinWeissVariable}
Fix $0<s<1$. Let $p(\cdot),q(\cdot)\in\mathcal{P}^{\log}(\mathbb{R}^d)$ with
$1<p^-\le p^+<\infty$ and $1<q^-\le q^+<\infty$. Assume
$p(x)\le q(x)$ for every  $x\in\mathbb{R}^d$.
Let  $a\in \mathbb{R}$ satisfy
\begin{equation}\label{eq:SW1}
-\frac{d}{q^+}<a\le 0,
\end{equation}
and suppose  that
\begin{equation}\label{eq:SW2}
\frac{1}{p(x)}-\frac{1}{q(x)}
=
\frac{s}{d}
+
\frac{a}{d}, \qquad x\in \rn.
\end{equation}
Then, for all
$u\in  X^{s,p(\cdot)}(\mathbb{R}^d)$, there holds
\[
\bigl\||\cdot|^{\,a} u\bigr\|_{L^{q(\cdot)}(\rn)}
\le
C
\bigl\|( \nabla^su)\bigr\|_{L^{p(\cdot)}(\rn;\rn)}.
\]
\end{theorem}

\begin{proof}
The inequality for $u\in C_0^\infty(\mathbb{R}^d)$ is a consequence of the
representation given in Theorem~\ref{thm-frac-FTOC}, which allows us to write
\(
u = I_s \sum_{j=1}^d \mathcal{R}_j \frac{\partial^s u}{\partial x_j^s},
\)  aside  from
\cite[Thm.~3.1]{CUS24} with $b=0$, and the boundedness of the Riesz-transform on $L_{p(\cdot)}(\rn)$ from  \cite[Lem.~3.3]{GHN07} (see also \cite[Lem.~12.4.3]{DHHR}). The general case follows from density.
\end{proof}

\begin{remark}
In particular, choosing $p(\cdot)=q(\cdot)$, $a=-s$
with \( s p^+ < d \) we obtain
\[
\bigl\||\,\cdot\,|^{-s}u\bigr\|_{L^{p(\cdot)}(\mathbb R^d)}
\leq
C\|\nabla^s u\|_{L^{p(\cdot)}
(\mathbb R^d;\mathbb R^d)}.
\]
\end{remark}

Our results obtained so far demonstrate a close analogy between the fractional
gradient and the classical gradient. So far, we have studied the
'subcritical' regime
\(
sp^+<d.
\)
We now turn to the \emph{limiting case}
\(
sp^-=d.
\)
We emphasize, however, that in the
variable exponent setting the condition $sp^-=d$ is only a sufficient global
condition and does not represent a sharp critical threshold.
In
particular, the intermediate regime
\(
sp^-<d\leq sp^+
\)
is not covered by these two cases. We  establish a Trudinger-type inequality in this limiting case which can be interpreted
as the critical Sobolev embedding
\(
X^{s,p(\cdot)}(\mathbb R^d)
\hookrightarrow
\exp L^{(p^-)'}(\Omega),
\)
where $\exp L^{p'_-}(\Omega)$ denotes the exponential Orlicz (or exponential
Zygmund) space. Consequently, the Trudinger inequality provides the endpoint
analogue of the Sobolev embedding theorem in the critical case.

\begin{theorem}[Variable exponent fractional Trudinger inequality]
Let $p(\cdot)\in\mathcal P^{\log}(\mathbb R^d)$ satisfy
\(
1<p^-\le p^+<\infty,
\) and
\(sp^{-}=d.
\)
Then there exist constants $A_1,A_2,C>0$ such that for every bounded open set
$\Omega\subset\mathbb R^d$  and every
$u\in X^{s,p(\cdot)}(\mathbb R^d)$ for which $g:=
\sum_{j=1}^d
\mathcal{R}_j\frac{\partial^su}{\partial x_j^s}$ is nonnegative, it holds
\[
\int_\Omega
\exp\left[
\left(
\frac{|u(x)|}
{A_1C\|u\|_{X^{s,p(\cdot)}(\mathbb R^d)}}
\right)^{(p^-)'}
\right]
\,dx
\le
A_2.
\]
\end{theorem}

\begin{proof}
By density it is enough to consider
$u\in C_0^\infty(\mathbb R^d)$.
{Theorem \ref{thm-frac-FTOC}}
yields
\[
u= I_s g, \qquad \text{where}\quad
g=
\sum_{j=1}^d
\mathcal{R}_j\frac{\partial^su}{\partial x_j^s}.
\]
Since $g$ is nonnegative, we can apply  the variable exponent Trudinger inequality for the Riesz
potential from \cite[Thm. 2, p. {627}]{FM08}, which gives
\[
\int_\Omega
\exp\left[
\left(
\frac{|u(x)|}
{A_1\|g\|_{L^{p(\cdot)}(\mathbb R^d)}}
\right)^{(p^-)'}
\right]
dx
\le
A_2.
\]

Since the Riesz transforms are bounded on
$L^{p(\cdot)}(\mathbb R^d)$, see \cite[Lem.~12.4.3]{DHHR},
we obtain
\begin{equation}\label{eq:est-g-Riesz}
\begin{aligned}
\|g\|_{L^{p(\cdot)}(\rn)}
=
\left\|
\sum_{j=1}^d
\mathcal{R}_j\frac{\partial^su}{\partial x_j^s}
\right\|_{L^{p(\cdot)}(\rn)}
\le
c
\sum_{j=1}^d
\left\|
\frac{\partial^su}{\partial x_j^s}
\right\|_{L^{p(\cdot)}(\rn)}
\le
C
\|\nabla^su\|_{L^{p(\cdot)}(\mathbb R^d;\rn)}\leq C \|u\|_{X^{s,p(\cdot)}(\rn)}.
\end{aligned}
\end{equation}

Consequently,
\[
\exp\left[
\left(
\frac{|u(x)|}
{A_1C\|u\|_{X^{s,p(\cdot)}(\rn)}}
\right)^{(p^-)'}
\right]
\le
\exp\left[
\left(
\frac{|u(x)|}
{A_1\|g\|_{L^{p(\cdot)}(\rn)}}
\right)^{(p^-)'}
\right].
\]

Integrating and using the previous estimate yields
\[
\int_\Omega
\exp\left[
\left(
\frac{|u(x)|}
{A_1C\|u\|_{X^{s,p(\cdot)}(\rn)}}
\right)^{(p^-)'}
\right]
dx
\le
A_2.
\]
\end{proof}

Finally, when \( s p^- > d \), we have the following analogy to the classical embedding theorem of Morrey.

\begin{theorem}[Variable exponent fractional Morrey type inequality]
Let $0<s<1$ and let
$p(\cdot)\in\mathcal P^{\log}(\mathbb R^d)$ satisfy
$
sp^->d
$ (implying $p^->1$) and $p^+<\infty$.
Then there exists a constant $C>0$ such that
\[
\sup_{y\in B(x,1)}
\frac{|u(x)-u(y)|}
{|x-y|^{\,s-\frac{d}{p(x)}}}
\le
C
\|\nabla^su\|_{L^{p(\cdot)}(\mathbb R^d;\rn)}
\]
for every
$u\in X^{s,p(\cdot)}(\mathbb R^d)$ and every
$x\in\mathbb R^d$.
\end{theorem}

\begin{proof}
Set
\(
\alpha(x):=s-\frac{d}{p(x)}\). Then
\(\alpha^{-} {=}s-\frac{d}{p^-}>0
\)
implying
\(
0<\alpha^{\color{cyan}-}\leq \alpha(x)<s<1.
\)
We first prove the assumption for
\(u\in C_0^\infty(\mathbb R^d)\). Using the fractional fundamental theorem of calculus (Theorem \ref{thm-frac-FTOC}) we write
\[
u= I_s g, \qquad \text{where}\quad
g=
\sum_{j=1}^d
\mathcal{R}_j\frac{\partial^su}{\partial x_j^s}.
\]
Since $g$ decays like $|z|^{-d-s}$ at infinity for $u\in C_0^{\infty}(\rn)$  which makes $I_sg$ a continuous function,   this is a pointwise identity, and
 for \(x\neq y\) we have
\[
u(x)-u(y)
=
\gamma(d,s)\int_{\mathbb R^d}
\left(
|x-z|^{s-d}-|y-z|^{s-d}
\right)g(z)\,dz.
\]

We use the generalized H\"older inequality from \cite[Lem.~3.2.20]{DHHR},
\[
\int_E|g(z)|\,dz
\leq
2\|g\|_{L^{p(\cdot)}(\mathbb R^d)}
 \|\chi_E\|_{L^{p'(\cdot)}(\mathbb R^d)},
\]
{where $E$ is any measurable subset of $\mathbb R^d$}, together with  the standard
characteristic-function estimate
\begin{equation}\label{eq:characteristic-small-ball}
\|\chi_{B(a,r)}\|_{L^{p'(\cdot)}(\mathbb R^d)}
\leq
C_0r^{d/p'(a)},
\qquad 0<r\leq1,
\end{equation}
which holds uniformly with respect to \(a\in\mathbb R^d\), see \cite[Cor.~4.5.9 and above]{DHHR}. Let
\(
h:=|x-y|\leq\frac14.
\)
We split the integral into a near part over \(B(x,2h)\) and a
far part over its complement.
For \(a\in\{x,y\}\), consider
\[
J(a,h)
:=
\int_{B(a,3h)}
|a-z|^{s-d}|g(z)|\,dz.
\]
Decompose \(B(a,3h)\) into the dyadic annuli
\[
A_k(a)
:=
\left\{
z:
2^{-k-1}3h\leq |a-z|<2^{-k}3h
\right\},
\qquad k\geq0.
\]
Using the generalized H\"older inequality and
\eqref{eq:characteristic-small-ball}, we obtain
\begin{align*}
J(a,h)
&\leq
C_1\sum_{k=0}^\infty
(2^{-k}h)^{s-d}
\int_{A_k(a)}|g(z)|\,dz
\\
&\leq
C_2\|g\|_{L^{p(\cdot)}{(\mathbb R^d)}}
\sum_{k=0}^\infty
(2^{-k}h)^{s-d}
(2^{-k}h)^{d/p'(a)}
=
C_2\|g\|_{L^{p(\cdot)}{(\mathbb R^d)}}
\sum_{k=0}^\infty
(2^{-k}h)^{s-d/p(a)}
\leq
C_3 h^{\alpha(a)}
\|g\|_{L^{p(\cdot)}{(\mathbb R^d)}}.
\end{align*}
The last series converges uniformly because
\(\alpha(a)\geq\alpha^{-}>0\). Since by definition of $h$ it holds
\[
B(x,2h)\subset B(y,3h),
\]
the near part therefore satisfies
\begin{align*}
\int_{B(x,2h)}
\left(
|x-z|^{s-d}+|y-z|^{s-d}
\right)|g(z)|\,dz
\leq
C_4\left(h^{\alpha(x)}+h^{\alpha(y)}\right)
\|g\|_{L^{p(\cdot)}{(\mathbb R^d)}}.
\end{align*}
Moreover, local logarithmic H\"older continuity gives
\[
\left|\frac1{p(x)}-\frac1{p(y)}\right|
\leq
\frac{c_{\log}}{\log(e+1/h)}.
\]
and consequently,
\begin{align*}
\frac{h^{\alpha(y)}}{h^{\alpha(x)}}
=
h^{d(1/p(x)-1/p(y))}
\leq
\exp\left(
d\left|\frac1{p(x)}-\frac1{p(y)}\right|
\log\frac1h
\right)
\leq e^{dc_{\log}}.
\end{align*}
Thus, the near part is bounded by
\begin{equation}\label{eq:near-part}
\int_{B(x,2h)}
\left(
|x-z|^{s-d}+|y-z|^{s-d}
\right)|g(z)|\,dz\leq C_5 h^{\alpha(x)}
\|g\|_{L^{p(\cdot)}(\mathbb R^d)}.
\end{equation}

For \(z\notin B(x,2h)\), the mean value theorem yields
\begin{equation}\label{MVT}
\left|
|x-z|^{s-d}-|y-z|^{s-d}
\right|
\leq
C_6 h|x-z|^{s-d-1}.
\end{equation}
Here we used that  every point $\xi$ on the line segment joining \(x\) and \(y\)
has distance
\[
|\xi-z|=|x+t(y-x)-z|
\geq |x-z|-t|x-y|
\geq \frac12|x-z|
\]
for every
\(t\in[0,1]\) and the fact that $s-d-1<0$.
First consider the region
\(
2h\leq |x-z|<\frac12.
\)
Using dyadic annuli $A_k$ with radii
\(r_k=2^kh\), i.e.,
\[
A_k
:=
\left\{
z\in\mathbb R^d:
2^k h\leq |x-z|<2^{k+1}h
\right\}
\cap B(x,1/2),
\qquad k\geq 1,
\]
the fact that $|x-z|\sim 2^kh$ on $A_k$,  and the generalized H\"older inequality, we further obtain
\begin{align*}
\int_{\substack{2h\leq |x-z|<1/2}}
\left|
|x-z|^{s-d}-|y-z|^{s-d}
\right||g(z)|\,dz
& \leq C_7\sum_{\substack{k\geq1\\2^kh<1/2}}
 h(2^k h)^{s-d-1}
\int_{A_k}|g(z)|\,dz
\\
&\leq
C_8\|g\|_{L^{p(\cdot)}{(\mathbb R^d)}}
\sum_{\substack{k\geq1\\2^kh<1/2}}
h(2^kh)^{s-d-1+d/p'(x)}\\
&=
C_8\|g\|_{L^{p(\cdot)}(\mathbb R^d)}
\sum_{\substack{k\geq1\\2^kh<1/2}}
h(2^kh)^{\alpha(x)-1}.
\end{align*}
Here the second step is again a consequence of the characteristic function estimate \eqref{eq:characteristic-small-ball} because
\(
A_k\subset B(x,2^{k+1}h)
\)
and \(2^{k+1}h<1\).
Since \(\alpha(x)<s<1\), the geometric series is bounded by
\(C_9h^{\alpha(x)}\). Hence
\begin{equation}\label{eq:intermediate-part}
\int_{\substack{2h\leq |x-z|<1/2}}
\left|
|x-z|^{s-d}-|y-z|^{s-d}
\right||g(z)|\,dz
\leq
C_{10} h^{\alpha(x)}
\|g\|_{L^{p(\cdot)}(\mathbb R^d)}.
\end{equation}

It remains to consider \(|x-z|\geq1/2\). According to \cite[Lem.~3.2.12]{DHHR} for measurable sets
\(E\) with \(|E|\geq1\), one has
\begin{equation}\label{eq:characteristic-large-sets}
\|\chi_E\|_{L^{p'(\cdot)}{(\mathbb R^d)}}
\leq
\tilde{C}|E|^{1/(p')^{-}},
\qquad
{(p')}^{-}=(p^+)'.
\end{equation}
We consider dyadic annuli of radii \(2^k\),  i.e.,
\[
A_k:=
\left\{
z\in\mathbb R^d:
2^k\leq |x-z|<2^{k+1}
\right\}, \qquad k\geq 0
\]
together with the set $A_{-1}:=
\left\{
z\in\mathbb R^d:
\frac12\leq |x-z|<1
\right\}$.
Since \(A_k\subset B(x,2^{k+1})\) and
\(|A_k|\leq C_{11}2^{kd}\), estimate
\eqref{eq:characteristic-large-sets} gives
\[
\|\chi_{A_k}\|_{L^{p'(\cdot)}(\mathbb R^d)}
\leq
C_{12}2^{kd/(p')^{-}}.
\]
The finitely many annuli for which \(|A_k|<1\) can be absorbed into
the constant. Therefore, since
$A_1$  has uniformly bounded measure, using \eqref{MVT} and the generalized H\"older inequality yield
\begin{align*}
\int_{|x-z|\geq1/2}
\left|
|x-z|^{s-d}-|y-z|^{s-d}
\right||g(z)|\,dz
& =C_6 h\sum_{k=-1}^{\infty}\int_{A_k}|x-z|^{s-d-1}|g(z)|dz
\\
&\leq
C_{13}h\|g\|_{L^{p(\cdot)}(\mathbb R^d)}
\left(
1+\sum_{k=0}^\infty
2^{k(s-d-1+d/(p')^-)}
\right)
\\
&=
C_{13}h\|g\|_{L^{p(\cdot)}(\mathbb R^d)}
\left(
1+\sum_{k=0}^\infty
2^{k(s-1-d/p^+)}
\right).
\end{align*}
The series converges because
\(
s-1-\frac{d}{p^+}<0.
\)
Since \(0<h<1\) and \(\alpha(x)<1\), we have
\(h\leq h^{\alpha(x)}\). Thus
\begin{equation}\label{eq:far-part}
\int_{|x-z|\geq1/2}
\left|
|x-z|^{s-d}-|y-z|^{s-d}
\right||g(z)|\,dz
\leq
C_{14} h^{\alpha(x)}
\|g\|_{L^{p(\cdot)}(\mathbb R^d)}.
\end{equation}

Combining \eqref{eq:near-part},
\eqref{eq:intermediate-part}, and \eqref{eq:far-part}, we find
\begin{equation}\label{est-1/4}
|u(x)-u(y)|
\leq
C_{15}|x-y|^{s-d/p(x)}
\|g\|_{L^{p(\cdot)}(\mathbb R^d)}
\end{equation}
whenever \(|x-y|\leq1/4\). Suppose now that \(1/4<|x-y|<1\), and set
\[
x_j:=x+\frac{j}{4}(y-x),\qquad j=0,\ldots,4.
\]
Then \(|x_{j+1}-x_j|=|x-y|/4<1/4\). Applying the preceding estimate \eqref{est-1/4} to each consecutive pair and using \(\alpha(x_j)=s-d/p(x_j)>0\), we obtain
\[
\begin{aligned}
|u(x)-u(y)|
&\le \sum_{j=0}^3|u(x_{j+1})-u(x_j)|\le C_{15} \sum_{j=0}^3\left(\frac{|x-y|}{4}\right)^{\alpha(x_{j})}\|g\|_{L^{p(\cdot)}(\mathbb R^d)}\leq 4C_{15}\|g\|_{L^{p(\cdot)}(\mathbb R^d)}.
\end{aligned}
\]
Moreover, since \(1/4<|x-y|<1\) and \(0<\alpha(x)<s\), we have
\(
|x-y|^{\alpha(x)}\ge 4^{-s} \iff    1\leq 4^s|x-y|^{\alpha(x)}.
\)
Therefore,
\[
|u(x)-u(y)|
\le
4^{s+1}C_{15}\,|x-y|^{s-d/p(x)}
\|g\|_{L^{p(\cdot)}(\mathbb R^d)}.
\]

As in \eqref{eq:est-g-Riesz} we have that
\(
\|g\|_{L^{p(\cdot)}(\mathbb R^d)}
\leq
C'\|\nabla^su\|_{L^{p(\cdot)}(\mathbb R^d;\mathbb R^d)}.
\)
Therefore,
\begin{equation}\label{eq:smooth-morrey}
|u(x)-u(y)|=|I_sg(x)-I_sg(y)|
\leq
C''|x-y|^{s-d/p(x)}
\|\nabla^su\|_{L^{p(\cdot)}(\mathbb R^d;\mathbb R^d)}
\end{equation}
for every \(u\in C_0^\infty(\mathbb R^d)\) and
\(0<|x-y|<1\).
Let now \(u\in X^{s,p(\cdot)}(\mathbb R^d)\). By density, choose
\(u_k\in C_0^\infty(\mathbb R^d)\) such that
\(
u_k\to u
 \text{ in }X^{s,p(\cdot)}(\mathbb R^d).
\)
We show that \((u_k)_k\) converges locally uniformly. Let \(Q\) be a ball
of radius at most \(1/4\), set \(v_{k,\ell}:=u_k-u_\ell\), and write
\[
(v_{k,\ell})_Q
:=
\frac1{|Q|}\int_Qv_{k,\ell}(z)\,dz,
\qquad
\operatorname*{osc}_Qv_{k,\ell}
:=
\sup_{z,w\in Q}|v_{k,\ell}(z)-v_{k,\ell}(w)|.
\]
By \eqref{eq:smooth-morrey}, using $|z-w|^{s-d/p(x)}\leq 1$ we obtain
\[
\operatorname*{osc}_Qv_{k,\ell}
\leq
C''\|\nabla^sv_{k,\ell}\|_{L^{p(\cdot)}(\rn;\rn)}.
\]
Moreover, the generalized H\"older  inequality gives
\[
|(v_{k,\ell})_Q|
\leq
\frac{2}{|Q|}
\|v_{k,\ell}\|_{L^{p(\cdot)}(\rn)}
\|\chi_Q\|_{L^{p'(\cdot)}(\rn)}.
\]
Since
\[
\|v_{k,\ell}\|_{L^\infty(Q)}
\leq \sup_{z\in Q}|v_{k,\ell}(z)-(v_{k,\ell})_Q|+|(v_{k,\ell})_Q|
\leq
\operatorname*{osc}_Qv_{k,\ell}
+|(v_{k,\ell})_Q|,
\]
it follows that
\[
\|v_{k,\ell}\|_{L^\infty(Q)}
\leq
C_Q\|v_{k,\ell}\|_{L^{p(\cdot)}(\rn)}
+
C''\|\nabla^sv_{k,\ell}\|_{L^{p(\cdot)}(\rn)}.
\]
Thus \((u_k)\) is {a} locally uniformly Cauchy {sequence}  and  {consequently}  converges to a continuous
function \(\widetilde u\). Since \(u_k\to u\) in \(L^{p(\cdot)}(\rn)\),
we have \(\widetilde u=u\) almost everywhere. Finally, passing to the limit in \eqref{eq:smooth-morrey} yields
\[
|\widetilde u(x)-\widetilde u(y)|
\leq
C''|x-y|^{s-d/p(x)}
\|\nabla^su\|_{L^{p(\cdot)}(\mathbb R^d;\mathbb R^d)}
\]
for \(0<|x-y|<1\).  Rearranging and taking the supremum over
\(y\in B(x,1)\setminus\{x\}\) proves the assertion.
\end{proof}

We obtain the following Morrey estimate for the Riesz potential which is of interest on its own.

\begin{corollary}[Variable exponent Morrey estimate for the Riesz potential]
    Let $0<s<1$ and let
$p(\cdot)\in\mathcal P^{\log}(\mathbb R^d)$ satisfy
$
sp^->d
$  and $p^+<\infty$.  Moreover,suppose  $g\in L^{p(\cdot)}(\rn)$ and that  $I_s\ast g$ is well-defined.
Then there exists a constant $C>0$ such that for every
$x\in\mathbb R^d$ it holds
\[
\sup_{y\in B(x,1)}
\frac{|I_s g(x)-I_s g(y)|}
{|x-y|^{\,s-\frac{d}{p(x)}}}
\leq
C
\|g\|_{L^{p(\cdot)}(\mathbb R^d)}.
\]
\end{corollary}

\begin{remark}
The requirement that $I_{s}\ast g$ be well-defined can be made precise as follows: we require  $I_{s}\ast g(x)<\infty$ for every \ $x\in\mathbb R^d$. This holds, in particular, whenever $g$ has compact support, or, more generally, whenever $g\in L^r(\mathbb R^d)$ for some $1\le r<\frac{d}{s}$.
\end{remark}

\section{Fractional Calculus of Variations and PDEs}
\label{sect-4}

 As an application of our nonlocal Sobolev spaces with variable integrability,
we now study a new class of fractional PDEs
which are obtained by minimizing variational problems in our nonlocal fractional Sobolev spaces with variable integrability.

\subsection{Lower weak semicontinuity in variable Lebesgue spaces}

The aim of this section is to establish weak lower semicontinuity of functionals of the form
\[
I(u)=\int_{\mathbb R^d}f(x,\nabla^su(x))\,dx,\qquad u\in X^{s,p(\cdot)}(\mathbb R^d),
\]
where $f:\mathbb R^d\times\mathbb R^d\to\mathbb R\cup\{+\infty\}$ is Borel measurable and, for a.e.\ $x\in\mathbb R^d$, the map $\xi\mapsto f(x,\xi)$ is lower semicontinuous, i.e.\
\[
\liminf_{\nu\to\infty}f(x,\xi_\nu)\ge f(x,\xi)\qquad\text{whenever }\xi_\nu\to\xi.
\]

The main result of this section shows that, under suitable growth assumptions on $f$ together with convexity of $\xi\mapsto f(x,\xi)$, the functional $I$ is weakly lower semicontinuous on $X^{s,p(\cdot)}(\mathbb R^d)$, that is,
\[
u_\nu\rightharpoonup u\ \text{ in }X^{s,p(\cdot)}(\mathbb R^d)
\quad\Longrightarrow\quad
\liminf_{\nu\to\infty}I(u_\nu)\ge I(u).
\]

We start with a result that allows a passage from weak to
strong convergence and can be found in  \cite[Thm.~3.13]{Rud91}. It is a direct consequence of the Hahn-Banach theorem.

\begin{theorem}[Mazur's Theorem]\label{thm-marzur}
Let $(X,\|\cdot\|)$ be a normed space and let
\[
x_\nu \rightharpoonup x \quad \text{in } X.
\]
Then there exists a sequence $\{y_\mu\}_{\mu=1}^\infty \subset \operatorname{co}\{x_\nu\}_{\nu=1}^\infty$ such that
\[
y_\mu \to x \quad \text{in } X,
\]
where $\operatorname{co}\{x_\nu\}_{\nu=1}^\infty$ denotes the convex hull of $\{x_\nu\}_{\nu=1}^\infty$, i.e.\ the set of all finite convex combinations of elements of $\{x_\nu\}_{\nu=1}^\infty$.
\end{theorem}

We will use Mazur's Theorem in order to prove the following theorem on lower weak semicontinuity in variable Lebesgue spaces.

\begin{theorem}\label{thm-low-semicont}
Let $\Omega\subset\mathbb{R}^d$ be an open set and let
$p(\cdot)\in\mathcal{P}^{\log}(\Omega)$ satisfy
\(
1<p^- \le p^+<\infty.
\)
Let
\(
f:\Omega\times\mathbb R^d\to\mathbb R\cup\{+\infty\}
\)
be a Borel measurable function  satisfying
\bit
\item \(
    \xi \mapsto f(x,\xi)
    \)
      is lower semicontinuous for almost every $x \in \Omega$,
\item  \(\xi\mapsto f(x,\xi)\) is convex,
\item  the growth condition
\[
f(x,\xi)\ge \langle a(x),\xi\rangle+b(x)
\]
for almost every $x\in\Omega$ and every $\xi\in\mathbb R^d$, where
\(
a\in L^{p'(\cdot)}(\Omega;\mathbb R^d),
\
b\in L^1(\Omega),
\)
and $\langle\cdot,\cdot\rangle$ denotes the Euclidean scalar product.
\eit
Then the functional
\[
J(\xi):=\int_\Omega f(x,\xi(x))\,dx
\]
  is weakly lower semicontinuous in $L^{p(\cdot)}(\Omega;\mathbb R^d)$, i.e., if
\(
\xi_\nu\rightharpoonup\xi
\ \text{in }L^{p(\cdot)}(\Omega;\mathbb R^d)\)
then
\[\liminf_{\nu\to\infty}J(\xi_\nu)\ge J(\xi).
\]
\end{theorem}

\begin{remark}
A function $f:\Omega\times\mathbb R^d\to\mathbb R\cup\{+\infty\}$ satisfying the
hypotheses of Theorem \ref{thm-low-semicont}, i.e.\ jointly Borel measurable in
$(x,\xi)$ and, for almost every $x\in\Omega$, lower semicontinuous and convex in
$\xi$, is commonly called a \emph{normal convex integrand}
(see e.g.\ \cite[Ch.~6]{FL07}). The joint measurability of $f$ is exactly what is needed to
ensure that $x\mapsto f(x,\xi(x))$ is Lebesgue measurable for every measurable
$\xi:\Omega\to\mathbb R^d$, so that the functional $J(\xi)=\int_\Omega
f(x,\xi(x))\,dx$ in Theorem \ref{thm-low-semicont} is well defined with values
in $\mathbb R\cup\{+\infty\}$.  If, in addition, $\xi\mapsto f(x,\xi)$ is continuous (rather
than merely lower semicontinuous) for a.e.\ $x$, one calls $f$ a
\emph{Carath\'eodory function}; in that case the measurability of
$x\mapsto f(x,\xi(x))$ is immediate. Note that  our class additionally allows $f(x,\cdot)$ to take the value $+\infty$, which continuity would rule out.
\end{remark}

\begin{proof}
We divide the proof into three steps.

\medskip

\noindent
\textbf{Step 1.}
Let
\[
h(x,\xi):=\langle a(x),\xi\rangle+b(x).
\]
Since
\(
a\in L^{p'(\cdot)}(\Omega;\mathbb R^d),
\
b\in L^1(\Omega),
\)
$h$ is affine-linear in $\xi$,
and
\(
\xi_\nu\rightharpoonup\xi
\ \text{in }
L^{p(\cdot)}(\Omega;\mathbb R^d),
\)
it follows from the definition of weak convergence that
\[
\lim_{\nu\to\infty}
\int_\Omega h(x,\xi_\nu(x))\,dx
=
\int_\Omega h(x,\xi(x))\,dx.
\]

Define
\(
g(x,\xi):=f(x,\xi)-h(x,\xi).
\)
Then \(g\ge0\) and since
$h$ is affine and continuous in
$\xi$,
$g$ inherits lower semicontinuity and convexity in
$\xi$ from
$f$.  Therefore, if we prove that
\[
\liminf_{\nu\to\infty}
\int_\Omega g(x,\xi_\nu(x))\,dx
\ge
\int_\Omega g(x,\xi(x))\,dx,
\]
the theorem follows. Hence, without loss of generality, we may assume
\(f\ge0\).

\medskip

\noindent
\textbf{Step 2.}
We next prove that \(J\) is strongly lower semicontinuous. Assume that
\(
\xi_\nu\to\xi\quad\text{in } L^{p(\cdot)}(\Omega;\mathbb R^d).
\)
Set \(\ell:=\liminf_{\nu\to\infty}J(\xi_\nu)\). By definition of the \(\liminf\), there exists a subsequence \(\{\xi_{\nu_k}\}_{k=1}^\infty\) such that
\[
J(\xi_{\nu_k})\longrightarrow \ell \quad\text{as } k\to\infty.
\]
Since \(\xi_{\nu_k}\to\xi\) in \(L^{p(\cdot)}(\Omega;\mathbb R^d)\) as well, passing to a further subsequence (not relabeled), we may assume that
\[
\xi_{\nu_k}(x)\to\xi(x)\quad\text{for a.e. }x\in\Omega,
\]
see \cite[Lem. 3.2.10(a)]{DHHR}. Since \(f\ge0\), Fatou's lemma yields
\[
\liminf_{k\to\infty}\int_\Omega f(x,\xi_{\nu_k}(x))\,dx
\ge
\int_\Omega \liminf_{k\to\infty} f(x,\xi_{\nu_k}(x))\,dx.
\]
Thus, using the lower semicontinuity of \(\xi\mapsto f(x,\xi)\) for a.e. \(x\in\Omega\), we obtain
\[
\liminf_{k\to\infty} J(\xi_{\nu_k})\ge J(\xi).
\]
Since \(\big(J(\xi_{\nu_k})\big)_k\) converges to \(\ell\), every subsequence of it converges to \(\ell\) as well, so that
\[
\ell = \lim_{k\to\infty}J(\xi_{\nu_k}) = \liminf_{k\to\infty}J(\xi_{\nu_k})\ge J(\xi).
\]
Recalling \(\ell=\liminf_{\nu\to\infty}J(\xi_\nu)\) proves
\[
\liminf_{\nu\to\infty}J(\xi_\nu)\ge J(\xi).
\]

\textbf{Step 3.}
We now pass from strong to weak lower semicontinuity using Mazur's Theorem.
{Assume that \linebreak
\(
\xi_\nu \rightharpoonup \xi\quad\text{in } L^{p(\cdot)}(\Omega;\mathbb R^d)
\) and set}
\[
L:=
\liminf_{\nu\to\infty}
J(\xi_\nu).
\]
Since \(f\ge0\), we have \(L\ge0\). We may assume that
\(L<\infty\), since otherwise there is nothing to prove.
Passing to a subsequence if necessary, we may further assume that
\[
L=
\lim_{\nu\to\infty}
J(\xi_\nu).
\]
Hence, for every \(\varepsilon>0\), there exists
\(\nu_\varepsilon\) such that
\[
J(\xi_\nu)\le L+\varepsilon
\qquad
\text{for every }
\nu\ge\nu_\varepsilon.
\]

Applying Theorem \ref{thm-marzur}  in
\(L^{p(\cdot)}(\Omega;\mathbb R^d)\),
there exists a sequence
\[
\{\eta_\mu\}_{\mu=1}^\infty
\subset
\operatorname{co}
\{\xi_\nu\}_{\nu=\nu_\varepsilon}^\infty
\]
such that
\(
\eta_\mu\to\xi
\ \text{strongly in }
L^{p(\cdot)}(\Omega;\mathbb R^d),
\)
and, for every \(\mu\),
\[
\eta_\mu
=
\sum_{i=\nu_\varepsilon}^{N_\mu}
\alpha_\mu^i\xi_i,
\qquad
\alpha_\mu^i \geq 0,
\qquad
\sum_{i=\nu_\varepsilon}^{N_\mu}
\alpha_\mu^i=1.
\]
Moreover, since \(f(x,\cdot)\) is convex, we deduce
\[
J(\eta_\mu)
=\int_\Omega f(x,\eta_{\mu}(x))\,dx
=\int_\Omega f\left(x,\sum_{i=\nu_\varepsilon}^{N_\mu}
\alpha_\mu^i\xi_i(x)\right)\,dx
\leq \sum_{i=\nu_\varepsilon}^{N_\mu}
\alpha_\mu^i\int_\Omega f (x,\xi_i(x))dx
=
\sum_{i=\nu_\varepsilon}^{N_\mu}
\alpha_\mu^i
J(\xi_i)
\le
L+\varepsilon.
\]

Combining this estimate with Step~2 yields
\[
J(\xi)
\le {\liminf_{\mu\to\infty}J(\eta_\mu) \le }
L+\varepsilon.
\]
Since \(\varepsilon>0\) is arbitrary, we get
\(
J(\xi)
\le
\liminf_{\nu\to\infty}
J(\xi_\nu),
\)
which proves the theorem.
\end{proof}

As an immediate corollary of Theorem \ref{thm-low-semicont} we get  the following result for functionals involving the fractional gradient. Indeed, since \(\nabla^s:X^{s,p(\cdot)}(\Omega)\to L^{p(\cdot)}(\Omega;\mathbb R^d)\) is a bounded linear operator, \(u_\nu\rightharpoonup u\) in \(X^{s,p(\cdot)}(\Omega)\) implies \(\nabla^su_\nu\rightharpoonup\nabla^su\) in \(L^{p(\cdot)}(\Omega;\mathbb R^d)\), so that Theorem \ref{thm-low-semicont} applies with \(\xi_\nu:=\nabla^su_\nu\).
{Note that the spaces on an open subset $\Omega \subset \mathbb{R}^n$ are understood to be defined by restriction.}

\begin{corollary} \label{cor-wls}
Let $\Omega\subset\mathbb{R}^d$ be an open set.
Suppose the assumptions of Theorem \ref{thm-low-semicont} on
$p$ and $f$ hold.
Then the functional
\(
I(u):=
\int_\Omega
f(x,\nabla^s u(x))\,dx
\)
 is weakly lower semicontinuous in $X^{s,p(\cdot)}(\Omega)$.
\end{corollary}

\subsection{A nonlocal nonlinear variational problem}

In the nonlinear setting, we first need to define appropriate energy
spaces.

\begin{definition}
Let $\Omega\subset\mathbb R^d$ be open, let
$p(\cdot)\in \mathcal{P}^{\log}(\mathbb R^d)$ satisfy
\(
1<p^- \le p^+<\infty,
\)
and let $\varphi\in X^{s,p(\cdot)}(\mathbb R^d)$.
We define
\[
X^{s,p(\cdot)}_{\varphi}(\Omega)
:=
\left\{
u\in X^{s,p(\cdot)}(\mathbb R^d)
: \
u=\varphi
\text{ in }\Omega^c
\right\}.
\]
\end{definition}

With the energy space defined, we can now state our main theorem concerning the
existence of minimizers of integral functionals of the fractional gradient.

\begin{theorem}[Existence Theorem]\label{thm-minimizer}
Let $p(\cdot)\in \mathcal{P}^{\log}(\mathbb R^d)$ such that
\(
1<p^-\le p^+<\infty
\), $s\in(0,1)$ so that $s<\frac{d}{p^+}$, and define {$p^{\ast}$ as in \eqref{def-p-ast}. Moreover let $\Omega\subset \rn$ be open and bounded,}
$q(\cdot)\in  \mathcal P(\Omega)$  be such that $1 < q^-\leq q(x) < p^{*}(x)$, {  $x\in \Omega$,} and let
$g\in L^{q'(\cdot)}(\Omega)$.
Let $\varphi\in X^{s,p(\cdot)}(\mathbb R^d)$ such that $I_{1-s} * \varphi$ is a well-defined locally integrable function.
Assume that
$f:\mathbb R^d\times\mathbb R^d\to\mathbb R$
 is a Borel measurable function, which is   lower semicontinuous for a.e. $x\in \rn$ (first variable) and convex in the second variable, and that there exist constants $c,C>0$ and functions
$\gamma_1,\gamma_2\in L^1(\mathbb R^d)$ satisfying
\begin{equation}\label{cond-f}
c|z|^{p(x)}+\gamma_1(x)
\le
f(x,z)
\le
\gamma_2(x)+C|z|^{p(x)}
\end{equation}
for all $(x,z)\in\mathbb R^d\times\mathbb R^d$.
Then the functional
\begin{equation}\label{functional-min}
F_s(u)
:=
\int_{\mathbb R^d}
f(x,\nabla^su)\,dx
-
\int_\Omega gu\,dx
\end{equation}
attains its minimum over
$X^{s,p(\cdot)}_{\varphi}(\Omega)$.
\end{theorem}

\begin{proof}
Define
\[
C_s
:=
\inf_{u\in X^{s,p(\cdot)}_{\varphi}(\Omega)}
\left(
\int_{\mathbb R^d}
f(x,\nabla^su)\,dx
-
\int_\Omega gu\,dx
\right).
\]
Since $\varphi\in X^{s,p(\cdot)}_\varphi(\Omega)$, we have $C_s\le F_s(\varphi)$, and using \eqref{cond-f} together with the generalized H\"older inequality we obtain
\begin{align*}
C_s
\le
F_s(\varphi)
&=
\int_{\mathbb R^d} f(x,\nabla^s\varphi)\,dx
-
\int_\Omega g\varphi\,dx
\\
&\le
\|\gamma_2\|_{L^1(\mathbb R^d)}
+
C\int_{\mathbb R^d}|\nabla^s\varphi|^{p(x)}\,dx
+
\int_\Omega |g\varphi|\,dx
\\
&\le
\|\gamma_2\|_{L^1(\mathbb R^d)}
+
C\int_{\mathbb R^d}|\nabla^s\varphi|^{p(x)}\,dx
+
2\|g\|_{L^{q'(\cdot)}(\Omega)}\,\|\varphi\|_{L^{q(\cdot)}(\Omega)}
\ <\ \infty,
\end{align*}
where the last line is finite since $\varphi\in X^{s,p(\cdot)}(\mathbb R^d)$ and $\varphi\in L^{q(\cdot)}(\Omega)$ by Theorem \ref{Iq}.
Hence there exists a minimizing sequence
$\{u_n\}\subset X^{s,p(\cdot)}_\varphi(\Omega)$ satisfying

\begin{equation}\label{Cs-limit}
C_s
=
\lim_{n\to\infty}
\left(
\int_{\mathbb R^d}
f(x,\nabla^su_n)\,dx
-
\int_\Omega gu_n\,dx
\right).
\end{equation}
We claim that the sequences $(u_n)$ and $(\nabla^s u_n)$ are  bounded in $L^{q(\cdot)}(\Omega)$ and
$L^{p(\cdot)}(\rn;\rn)$, respectively.
The first claim is a consequence of the second, by Theorem \ref{Iq}.
First we should note that, similarly as above and using Young's inequality with $\varepsilon>0$, it holds
\begin{eqnarray}
\int_{\Omega} g u_n \; dx
&\leq & 2 \|g\|_{L^{q'(\cdot)}(\Omega)}
\| u_n\|_{L^{q(\cdot)}(\Omega)}\notag\\
&\leq & c' \|g\|_{L^{q'(\cdot)}(\Omega)}
\|\nabla^s u_n\|_{L^{p(\cdot)}(\rn;\rn)}
\le  \varepsilon \, \|\nabla^s u_n\|_{L^{p(\cdot)}(\rn;\rn)}^{p^-}+  C(\varepsilon) \, \|g\|_{L^{q'(\cdot)}(\Omega)}^{(p^-)'} \label{help-ineq}
\end{eqnarray}
with positive constants   $c(\varepsilon)$ and $C(\varepsilon)$, such that $c(\varepsilon)$  can be chosen as small as we want.
Then, using \eqref{cond-f} we can deduce that, for any $n$ sufficiently large and in case $\|\nabla^su_n\|_{L^{p(\cdot)}(\rn;\rn)} >1$,
\[
\begin{aligned}
C_s+1
&\ge
c
\int_{\mathbb R^d}
|\nabla^su_n|^{p(x)}\,dx
+
\int_{\mathbb R^d}
\gamma_1(x)\,dx
-
\int_\Omega gu_n\,dx
\\
&\ge
c\;
\|\nabla^s u_n\|^{p^-}_{L^{p(\cdot)}(\rn;\rn)}
+
\|
\gamma_1  \|_{L^{1}(\rn)}
-\varepsilon \, \|\nabla^s u_n\|_{L^{p(\cdot)}(\rn;\rn)}^{p^-} - C(\varepsilon) \, \|g\|_{L^{q'(\cdot)}(\Omega)}^{(p^-)'},
\end{aligned}
\]
where the second step we used  \cite[Lem.~3.2.5]{DHHR} to estimate the first term on the right hand side and  \eqref{help-ineq} for the last.
Choosing $\varepsilon$ sufficiently small yields the boundedness of the sequence $(\nabla^s u_n)$ in $L^{p(\cdot)}(\mathbb{R}^d;\rn)$.

Since the spaces $L^{q(\cdot)}(\Omega)$ and
$L^{p(\cdot)}(\rn)$ are reflexive,
cf. \cite[Theorem 3.4.7]{DHHR},  the above mentioned bounded sequences admit weak convergent subsequences (see e.g. \cite[Theorem 3.18]{Brezis2011}).

Since $u_n=\varphi$ on $\Omega^c$ by definition, we have $\operatorname{supp}(u_n-\varphi)\subset\overline\Omega$, so that $(u_n-\varphi)$, extended by zero, is also bounded in $L^{q(\cdot)}(\rn)$. Hence, up to a subsequence, there exist $u\in L^{q(\cdot)}(\Omega)$ and $v\in L^{p(\cdot)}(\rn;\rn)$ such that
\[
u_n-\varphi
\rightharpoonup
u-\varphi
\quad\text{in }
L^{q(\cdot)}(\Omega) \quad \text{ and } \quad \nabla^su_n
\rightharpoonup
v
\quad\text{in }
L^{p(\cdot)}(\rn;\rn).
\]
We postpone asserting $u\in X^{s,p(\cdot)}_\varphi(\Omega)$ until we have shown $v=\nabla^su$ below; at this stage we only know $u\in L^{q(\cdot)}(\Omega)$, with $u=\varphi$ on $\Omega^c$.
To see that in fact $u\in L^{p(\cdot)}(\rn)$ -- needed for $\nabla^su$ to be well-defined per Definition~\ref{def-Xsp} -- note that since $u_n-\varphi\in X^{s,p(\cdot)}(\rn)$ with $\nabla^s(u_n-\varphi)$ bounded in $L^{p(\cdot)}(\rn;\rn)$, the global Sobolev-type inequality of Theorem~\ref{fractionalSI} gives
\[
\|u_n-\varphi\|_{L^{p^*(\cdot)}(\rn)}\le C\,\|\nabla^s(u_n-\varphi)\|_{L^{p(\cdot)}(\rn;\rn)}\le C',
\]
uniformly in $n$. Since $p(x)<p^*(x)$ for every $x$ and $\Omega$ is bounded, the embedding $L^{p^*(\cdot)}(\Omega)\hookrightarrow L^{p(\cdot)}(\Omega)$ shows that $(u_n-\varphi)$ is also bounded in $L^{p(\cdot)}(\Omega)$; identifying the limit of a further subsequence with $u-\varphi$, we conclude $u-\varphi\in L^{p(\cdot)}(\Omega)$, hence $u=\varphi+(u-\varphi)\in L^{p(\cdot)}(\rn)$.

It remains to show that $v=\nabla^su$,
after which weak lower semicontinuity yields the desired conclusion.  Precisely, if we can establish this claim, from \eqref{Cs-limit} we conclude
\[
\begin{aligned}
C_s=\lim_{n\to\infty}
\left(
\int_{\mathbb R^d}
f(x,\nabla^su_n)\,dx
-
\int_\Omega gu_n\,dx
\right)
&\ge
\liminf_{n\to\infty}
\int_{\mathbb R^d}
f(x,\nabla^su_n)\,dx
-
\limsup_{n\to\infty}
\int_\Omega gu_n\,dx
\\
&\ge
\int_{\mathbb R^d}
f(x,\nabla^su)\,dx
-
\int_\Omega gu\,dx .
\end{aligned}
\]

Here we have used the weak lower semicontinuity of the first term,  guaranteed by Theorem \ref{thm-low-semicont}, together
with the weak continuity of the linear functional
\(
u\mapsto\int_\Omega gu\,dx
\),
which follows from the duality
\(
\big(L^{q(\cdot)}(\Omega)\big)^*
=
L^{q'(\cdot)}(\Omega).
\)

It remains to prove that
\[
v=\nabla^su.
\]

For
\(
w\in C_0^\infty(\mathbb R^d)
\),
the definition of the fractional gradient gives, for $j\in\{1, \dots, d\}$,
\[
\left\langle
\frac{\partial^s}{\partial x_j^s}(u_n-\varphi),
w
\right\rangle
=
-
\int_{\mathbb R^d}
I_{1-s}*(u_n-\varphi)
\frac{\partial w}{\partial x_j}\,dx
=
-
\int_\Omega
(u_n-\varphi)
I_{1-s}*\frac{\partial w}{\partial x_j}\,dx,{\footnotemark}
\]
 since
\(
\operatorname{supp}(u_n-\varphi)\subset\overline{\Omega}
\)\textcolor{red}{\footnotetext{The order of integration may be exchanged by Fubini's theorem: since the kernel of the Riesz potential $
\kappa(x):=\gamma(d,1-s)\,|x|^{(1-s)-d}\in L^1_{\mathrm{loc}}(\mathbb R^d)$ (since $1-s>0$) and $\partial_{x_j}w$ is bounded with compact support, $\int_{\mathbb R^d}|\kappa(x-y)|\,|\partial_{x_j}w(x)|\,dx$ is bounded uniformly for $y$ in the bounded set $\overline\Omega$; together with $u_n-\varphi\in L^1(\Omega)$, this gives absolute convergence of the double integral.}}.
Passing to the limit using the weak convergence in
\(
L^{q(\cdot)}(\mathbb R^d)
\),
we obtain
\[
-
\lim_{n\to\infty}
\int_\Omega
(u_n-\varphi)
I_{1-s}*\frac{\partial w}{\partial x_j}\,dx
=
-
\int_{\mathbb R^d}
(u-\varphi)
I_{1-s}*\frac{\partial w}{\partial x_j}\,dx.{\footnotemark}
\]
\footnotetext{Since $\partial_{x_j}w\in C_0^\infty(\mathbb R^d)$, the function $I_{1-s}*\partial_{x_j}w$ is bounded and continuous on $\mathbb R^d$ (its Fourier transform $c|\xi|^{-(1-s)}\widehat{\partial_{x_j}w}(\xi)$ is $L^1$, being locally integrable near $\xi=0$ since $1-s<d$, and rapidly decaying at infinity). Restricted to the bounded set $\Omega$, it therefore belongs to $L^{q'(\cdot)}(\Omega)$, so the limit passes through by definition of the weak convergence $u_n-\varphi\rightharpoonup u-\varphi$ in $L^{q(\cdot)}(\Omega)$ and the duality $(L^{q(\cdot)}(\Omega))^*=L^{q'(\cdot)}(\Omega)$. The domain is then extended to $\mathbb R^d$ on the right using $u-\varphi=0$ a.e.\ on $\Omega^c$.}
Moving the Riesz potential back to the other factor yields
\[
\lim_{n\to\infty}
\left\langle
\frac{\partial^s}{\partial x_j^s}(u_n-\varphi),
w
\right\rangle
=
-
\int_{\mathbb R^d}
I_{1-s}*u
\frac{\partial w}{\partial x_j}\,dx
+
\int_{\mathbb R^d}
I_{1-s}*\varphi
\frac{\partial w}{\partial x_j}\,dx.
\]

On the other hand,
\[
\left\langle
\frac{\partial^s}{\partial x_j^s}(u_n-\varphi),
w
\right\rangle
=
\left\langle
\frac{\partial^s u_n}{\partial x_j^s},
w
\right\rangle
-
\left\langle
\frac{\partial^s\varphi}{\partial x_j^s},
w
\right\rangle.
\]

Since
\(
\partial_{x_j}^su_n
\rightharpoonup
v_j
\)
in
\(
L^{p(\cdot)}(\mathbb R^d)
\),
we have
\[
\lim_{n\to\infty}
\left\langle
\frac{\partial^s}{\partial x_j^s}(u_n-\varphi),
w
\right\rangle
=
\int_{\mathbb R^d}
v_jw\,dx
-
\left\langle
\frac{\partial^s\varphi}{\partial x_j^s},
w
\right\rangle.
\]

Comparing the two identities and cancelling the fractional derivatives of
\(
\varphi
\),
we conclude that
\[
v_j
=
\frac{\partial^su}{\partial x_j^s},
\]
for every
\(j=1,\dots,d\). Since $u\in L^{p(\cdot)}(\mathbb R^d)$ and $\nabla^su=v\in L^{p(\cdot)}(\mathbb R^d;\rn)$, we conclude $u\in X^{s,p(\cdot)}(\mathbb R^d)$; together with $u=\varphi$ on $\Omega^c$, this gives $u\in X^{s,p(\cdot)}_\varphi(\Omega)$, which completes the proof.
\end{proof}

\begin{remark}[Uniqueness]\label{rem-uniqueness}
If, in addition to the assumptions of Theorem~\ref{thm-minimizer}, the map
$z\mapsto f(x,z)$ is strictly convex for a.e.\ $x\in\mathbb R^d$, then the
minimizer of \eqref{functional-min} over $X^{s,p(\cdot)}_\varphi(\Omega)$ is
unique. Indeed, let $u_1,u_2$ be two minimizers. Since
$X^{s,p(\cdot)}_\varphi(\Omega)$ is convex, $w:=\frac12(u_1+u_2)\in
X^{s,p(\cdot)}_\varphi(\Omega)$, and $F_s(u_1)=F_s(u_2)=C_s$ is finite by
\eqref{cond-f}. If $\nabla^su_1\neq\nabla^su_2$ on a set of positive measure,
strict convexity of $f(x,\cdot)$ and linearity of the second term in
\eqref{functional-min} give $F_s(w)<\frac12\big(F_s(u_1)+F_s(u_2)\big)=C_s$,
a contradiction. Hence $\nabla^s(u_1-u_2)=0$ a.e., and since
$u_1-u_2\in X^{s,p(\cdot)}(\mathbb R^d)$, Theorem~\ref{fractionalSI} yields
$u_1=u_2$ a.e. 
\end{remark}

Having established the existence of a minimizer in Theorem \ref{thm-minimizer}, we now derive its weak Euler--Lagrange equation,
under the additional assumption that $f$ is continuously differentiable in its second variable.

\begin{theorem}[Euler--Lagrange equation]\label{thm-Euler-Lagrange}
Let $p(\cdot)\in\mathcal P^{\log}(\mathbb R^d)$ satisfy
\(
1<p^-\le p^+<\infty
\), $s\in(0,1)$, $s<\frac{d}{p^+}$, and define $p^{\ast}$ as in \eqref{def-p-ast} above. Let $\varphi\in X^{s,p(\cdot)}(\mathbb R^d)$ be such that $I_{1-s}*\varphi$ is a well-defined locally integrable function.
Moreover,
let $\Omega\subset \rn$ be open and bounded, let $q(\cdot)\in{\mathcal P}(\Omega)$ be such that $1 < q^-\leq q(x) < p^{*}(x)$, {$x\in \Omega$, and let
$g\in L^{q'(\cdot)}(\Omega)$.
Suppose that $f$ satisfies the assumptions of Theorem~\ref{thm-minimizer} and is, in addition, continuously differentiable in the second variable, and that there exist $a\in L^{p'(\cdot)}(\mathbb R^d)$ and $C_1>0$ such that
\[
|\nabla_zf(x,z)|\le a(x)+C_1|z|^{p(x)-1}
\]
for a.e.\ $x\in\mathbb R^d$ and every $z\in\mathbb R^d$.
If
\(
u\in X^{s,p(\cdot)}_\varphi(\Omega)
\)
is a minimizer of \eqref{functional-min}
then
\begin{equation}\label{Euler-Lagrange-eq}
\int_{\mathbb R^d}
\nabla_zf(x,\nabla^su)\cdot \nabla^sv\,dx
=
\int_\Omega gv\,dx
\end{equation}
for every
$v\in C_0^\infty(\Omega)$.}
\end{theorem}

\begin{proof}
Let $u$ be a minimizer of $F_s$, { defined in \eqref{functional-min},} and let
$v\in C_0^\infty(\Omega)$. {For $t\in\mathbb R$ let $I(t):=F_s({u+tv})$, which is finite for every $t\in\mathbb R$ by the upper bound in \eqref{cond-f}, exactly as in the $C_s<\infty$ argument in the proof of Theorem \ref{thm-minimizer}.}
Since
\(
u+tv\in X^{s,p(\cdot)}_\varphi(\Omega)
\)
for every $t\in\mathbb R$, the minimality of $u$ implies
\[
I(0)=\min_{t\in\mathbb R}I(t).
\]
Thus it suffices to prove that $I$ is differentiable at $t=0$.
We write
\[
\begin{aligned}
\frac{I(t)-I(0)}{t}
={}&
\int_{\mathbb R^d}
\frac{
f(x,\nabla^su+t\nabla^sv)-f(x,\nabla^su)
}{t}
\,dx
-
\int_\Omega gv\,dx.
\end{aligned}
\]

Since $f$ is continuously differentiable with respect to its second
variable, the  {fundamental theorem of calculus (applied to $\theta \mapsto f(x, \nabla^su+\theta t\nabla^s v)$)}
gives
\[
\frac{
f(x,\nabla^su+t\nabla^sv)-f(x,\nabla^su)
}{t}
=
\int_0^1
\nabla_zf\!\left(
x,\nabla^su+\theta t\nabla^sv
\right)
\cdot \nabla^sv
\,d\theta.
\]

Hence,
\[
\frac{I(t)-I(0)}{t}
=
\int_0^1
\int_{\mathbb R^d}
\nabla_zf\!\left(
x,\nabla^su+\theta t\nabla^sv
\right)
\cdot \nabla^sv
\,dx\,d\theta
-
\int_\Omega gv\,dx,
\]
{where Fubini's theorem applies since, as the estimate below shows, the integrand is dominated by an integrable function on $\mathbb R^d\times(0,1)$.}
We need to justify the passage to the limit $t\to0$.
By assumption,
\[
|\nabla_zf(x,z)|
\le
a(x)+{C_1}|z|^{p(x)-1},
\]
where
$a\in L^{p'(\cdot)}(\mathbb R^d)$.
Fix $|t|\le1$.
Since
\[
|\nabla^su+\theta t\nabla^sv|
\le
|\nabla^su|+|\nabla^sv|,
\]
using
\[
(|\xi|+|\eta|)^{p(x)-1}
\le
{ \max\{2^{p^+-2},1\}}
\left(
|\xi|^{p(x)-1}
+
|\eta|^{p(x)-1}
\right),
\]
we obtain
\[
\begin{aligned}
&
\left|
\nabla_zf\!\left(
x,\nabla^su+\theta t\nabla^sv
\right)
\cdot \nabla^sv
\right|
\le
\Bigl(
a(x)
+
{C_2}
\bigl(
|\nabla^su|^{p(x)-1}
+
|\nabla^sv|^{p(x)-1}
\bigr)
\Bigr)
|\nabla^sv|.
\end{aligned}
\]
The right-hand side belongs to $L^1(\mathbb R^d)$, since by the variable exponent Hölder inequality,
\[
\int_{\mathbb R^d}
a(x)|\nabla^sv|
\,dx
\le
2
\|a\|_{L^{p'(\cdot)}(\rn)}
\|\nabla^sv\|_{L^{p(\cdot)}(\rn)}.
\]

Furthermore, {by the definition of the conjugate exponent, it holds $(p(x)-1)\,p'(x)=p(x)$, thus,
\[
\rho_{p'(\cdot)}\!\left(|\nabla^su|^{p(\cdot)-1}\right)=\rho_{p(\cdot)}(\nabla^su)<\infty.
\]}
Hence
\(
|\nabla^su|^{p(\cdot)-1}\in L^{p'(\cdot)}(\mathbb R^d)
\)
and another application of Hölder's inequality yields
\[
\int_{\mathbb R^d}
|\nabla^su|^{p(x)-1}|\nabla^sv|
\,dx
<\infty,
\]
as well as
\[
\int_{\mathbb R^d}
|\nabla^sv|^{p(x)}
\,dx
=
\rho_{p(\cdot)}(\nabla^sv)
<\infty.
\]

Therefore the integrand is dominated by an $L^1$-function independent of
$t$.
Since
\[
\nabla_zf\!\left(
x,\nabla^su+\theta t\nabla^sv
\right)
\longrightarrow
\nabla_zf(x,\nabla^su)
\]
pointwise almost everywhere as $t\to0$,{\footnote{{Since $\theta=\theta(x,t)\in(0,1)$, we have $\theta t\,\nabla^sv(x)\to0$ as $t\to0$ for a.e.\ $x$, and $z\mapsto\nabla_zf(x,z)$ is continuous by assumption.}}} the dominated convergence theorem (applied on the product space $\rn\times (0,1)$; note that the limit of the integrand is independent of $\theta$)
gives
\[
\begin{aligned}
I'(0)
={}&
\int_{\mathbb R^d}
\nabla_zf(x,\nabla^su)\cdot \nabla^sv
\,dx
-
\int_\Omega gv\,dx.
\end{aligned}
\]

Finally, since $I$ attains its minimum at $t=0$, we conclude that
\(
I'(0)=0,
\)  which yields \eqref{Euler-Lagrange-eq}.
\end{proof}

As an immediate corollary, we now derive the solutions to a fractional $p(\cdot)$-Laplace equation within this setting.

\begin{corollary}[Fractional $p(\cdot)$-Laplace equation]
Let $p(\cdot)\in\mathcal P^{\log}(\mathbb R^d)$ satisfy $1<p^-\le p^+<\infty$, let $s\in(0,1)$ be such that $s<\frac{d}{p^+}$, and define $p^*$ as in \eqref{def-p-ast}. Let $\Omega\subset\mathbb R^d$ be open and bounded, and let $\varphi\in X^{s,p(\cdot)}(\mathbb R^d)$ be such that $I_{1-s}*\varphi$ is a well-defined locally integrable function. Let $q(\cdot)\in{\mathcal P(\Omega)}$ be such that $1<q^-\le q(x)<p^*(x)$ for $x\in\Omega$, and let $g\in L^{q'(\cdot)}(\Omega)$.
Then there exists $u\in X^{s,p(\cdot)}_\varphi(\Omega)$ such that
\[
\int_{\mathbb R^d}
|\nabla^su|^{p(x)-2}
\nabla^su\cdot \nabla^sv\,dx
=
\int_\Omega gv\,dx
\]
for every $v\in C_0^\infty(\Omega)$.
\end{corollary}

\begin{proof}
Let $f(x,z):=\frac{|z|^{p(x)}}{p(x)}$, so that $\nabla_zf(x,z)=|z|^{p(x)-2}z$. {We first check that $f$ satisfies the assumptions of Theorem~\ref{thm-minimizer}. Since $p\in\mathcal P^{\log}(\mathbb R^d)$ is continuous, $f$ is jointly continuous on $\mathbb R^d\times\mathbb R^d$, hence Borel measurable and, in particular, lower semicontinuous in $z$; moreover $z\mapsto|z|^{p(x)}$ is convex since $p(x)\ge1$, and
\[
\frac1{p^+}|z|^{p(x)}\le f(x,z)\le\frac1{p^-}|z|^{p(x)},
\]
so \eqref{cond-f} holds with $c=1/p^+$, $C=1/p^-$, $\gamma_1=\gamma_2\equiv0$. By Theorem~\ref{thm-minimizer}, the functional
\[
F_s(u)=\int_{\mathbb R^d}\frac{|\nabla^su(x)|^{p(x)}}{p(x)}\,dx-\int_\Omega gu\,dx
\]
attains a minimizer $u\in X^{s,p(\cdot)}_\varphi(\Omega)$.}
It remains to check the additional hypotheses of Theorem~\ref{thm-Euler-Lagrange}: $f$ is continuously differentiable in $z$, and
\(
|\nabla_zf(x,z)|=|z|^{p(x)-1},
\)
so the growth assumption on $\nabla_zf$ holds with $a\equiv0$ and $C_1=1$. Applying Theorem~\ref{thm-Euler-Lagrange} to the minimizer $u$ obtained above yields
\[
\int_{\mathbb R^d}
|\nabla^su|^{p(x)-2}
\nabla^su\cdot \nabla^sv\,dx
=
\int_\Omega gv\,dx
\]
for every $v\in C_0^\infty(\Omega)$.
\end{proof}

\begin{remark}
\bit
\item The  proof  of Theorem \ref{thm-Euler-Lagrange} differs from the corresponding one in \cite[Thm.~1.18]{SS18} by using  the dominated convergence theorem instead of G\^{a}teaux differentiability.
\item Note that (even in the constant exponent case) this is not to be confused with
existing notions of the fractional $p$-Laplace’s equation from the real interpolation
theory, see e.g. \cite{DCKP16,LL14}.
\item Although the exponent $p(x)$ depends on the spatial variable, no derivatives
of $p$ appear in the Euler--Lagrange equation. Indeed, the first variation is
taken with respect to the variable
\(
z=\nabla^su(x),
\)
while the point $x$ is regarded as a parameter. Consequently,
\(
\nabla_z\!\left(\frac{|z|^{p(x)}}{p(x)}\right)
=
|z|^{p(x)-2}z,
\)
exactly as in the constant exponent case. The dependence of $p$ on $x$
therefore influences only the growth of the operator and not its formal
Euler--Lagrange structure.
\eit
\end{remark}

\textbf{Acknowledgment:}
The authors acknowledge financial support by the Centre for Mathematics of the University of Coimbra (CMUC, https://doi.org/10.54499/UID/00324/2025) under the Portuguese Foundation for Science and Technology (FCT), Grants UID/00324/2025 and UID/PRR/00324/2025.
The second author gratefully acknowledges financial support from
the German Academic Exchange Service (DAAD).
Part of this work was carried out during a workshop supported by the joint AEI--DFG project ``Function Spaces and High-Dimensional Approximation.''


{\textbf{Declaration of AI use:} During the preparation of this work the authors used Claude Opus 5.5 (Anthropic) to improve the language and phrasing, to format and check the references (including DOIs), and to assist with the comparison to the existing literature. All mathematical content was developed and verified by the authors, who take full responsibility for the content of the publication.

\end{document}